\documentclass[reqno]{amsart}
\usepackage{amsmath, amssymb, amsthm, epsfig}
\usepackage{hyperref, latexsym}
\usepackage{url}
\usepackage[mathscr]{euscript}

\usepackage{color}
\usepackage{fullpage} 
\usepackage{setspace}

\def\today{\ifcase\month\or
  January\or February\or March\or April\or May\or June\or=
  July\or August\or September\or October\or November\or December\fi
  \space\number\day, \number\year}

 \newtheorem{theorem}{Theorem}
  
 \newtheorem{lemma}[theorem]{Lemma}
 \newtheorem{proposition}[theorem]{Proposition}
 
 \theoremstyle{definition}

 \theoremstyle{remark}

 \newcommand{\mc}{\mathcal}

 \newcommand{\R}{\mathbb{R}}
  \newcommand{\RR}{\mathcal{R}}
 \newcommand{\N}{\mathbb{N}}
 
 \newcommand{\Z}{\mathbb{Z}}

\newcommand{\dz}{\text{\rm d}z}
 \newcommand{\dt}{\text{\rm d}t}
 \newcommand{\du}{\text{\rm d}u}

 \newcommand{\dx}{\text{\rm d}x}
 
 \newcommand{\dy}{\text{\rm d}y}

\newcommand{\wt}{\widetilde}

\newcommand{\var}{{\rm Var\,}}

\begin{document}

\title[Derivative bounds for maximal operators]{Derivative bounds for fractional maximal functions}
\author[Carneiro and Madrid]{Emanuel Carneiro and Jos\'e Madrid}
\date{\today}

\address{IMPA - Instituto de Matem\'{a}tica Pura e Aplicada, Estrada Dona Castorina 110, Rio de Janeiro - RJ, Brazil, 22460-320.}
\email{carneiro@impa.br}
\email{josermp@impa.br}

\allowdisplaybreaks
\numberwithin{equation}{section}

\maketitle

\begin{abstract}In this paper we study the regularity properties of fractional maximal operators acting on $BV$-functions. We establish new bounds for the derivative of the fractional maximal function, both in the continuous and in the discrete settings.
\end{abstract}

\section{Introduction}

\subsection{Historical overview} Let $M$ denote the centered Hardy-Littlewood maximal operator on $\R^d$, i.e. for $f \in L^1_{loc}(\R^d)$,
\begin{equation}\label{Intro_max}
Mf(x) = \sup_{r >0} \frac{1}{m(B_r(x))} \int_{B_r(x)} |f (y)|\,\dy\,,
\end{equation}
where $B_r(x)$ is the ball centered at $x$ with radius $r$ and $m(B_r(x))$ is its $d$-dimensional Lebesgue measure. One of the cornerstones of harmonic analysis is the celebrated theorem of Hardy-Littlewood-Wiener that asserts that $M:L^p(\R^d) \to L^p(\R^d)$ is bounded for $1<p \leq \infty$. For $p=1$ we have $M: L^1(\R^d) \to L^{1,\infty}(\R^d)$ bounded. The process of averaging a function is, in essence, a variation-diminishing process (e.g. for a fixed radius $r$ in \eqref{Intro_max}), so it is natural to wonder if such behavior is preserved when taking a pointwise supremum over averages instead. This leads one to consider maximal operators acting on functions of bounded variation and Sobolev functions. The first result in this direction is due to Kinunnen \cite{Ki}, who showed that $M: W^{1,p}(\R^d) \to W^{1,p}(\R^d)$ is bounded for $1 < p \leq \infty$, elegantly combining basic tools of functional analysis. This paradigm has been extended to multilinear, local and fractional contexts in \cite{CM, KL, KiSa}. Due to the lack of reflexivity of $L^1$, results for $p=1$ are subtler. Examples of such results were motivated by the following question posed in \cite{HO}:

\smallskip

\noindent {\bf Question A.} (Haj\l asz and Onninen \cite{HO}) Is the operator $f \mapsto |\nabla Mf|$ bounded from $W^{1,1}(\R^d)$ to $L^1(\R^d)$?

\smallskip

\noindent A standard dilation argument reveals the true nature of this question: whether the variation of the maximal function is controlled by the variation of the original function, i.e. if we have
\begin{equation*}
\|\nabla Mf\|_{L^1(\R^d)} \leq C \|\nabla f\|_{L^1(\R^d)}.
\end{equation*}
Progress on this problem has been restricted to dimension $d=1$. For the uncentered maximal operator (defined similarly as in \eqref{Intro_max}, with the supremum taken over all balls containing the point $x$ in its closure), which we denote here by $\wt{M}$, Tanaka \cite{Ta} showed that if $f \in W^{1,1}(\R)$ then $\wt{M}f$ is weakly differentiable and 
\begin{equation}\label{Intro_Tanaka}
\big\|\big(\wt{M}f\big)'\big\|_{L^1(\R)} \leq 2\, \|f'\|_{L^{1}(\R)}.
\end{equation}
This result was later refined by Aldaz and P\'{e}rez L\'{a}zaro \cite{AP}, who showed that if $f$ is of bounded variation then $\wt{M}f$ is in fact absolutely continuous and 
\begin{equation}\label{Intro_AP}
\var \big(\wt{M}f\big) \leq \var (f),
\end{equation}
where $\var(f)$ denotes the total variation of $f$. Observe that inequality \eqref{Intro_AP} is sharp. More recently, in the remarkable paper \cite{Ku}, Kurka considered the centered maximal operator in dimension $d=1$ and proved that 
\begin{equation}\label{Intro_Ku}
\var (Mf) \leq 240,004\, \var (f).
\end{equation}
It is also shown in \cite{Ku} that if $f \in W^{1,1}(\R)$ then $Mf$ is weakly differentiable and \eqref{Intro_Tanaka} also holds with constant $C = 240,004$. It is currently unknown if one can bring down the value of such constant to $C=1$ in the centered case. The status of Question A in the general case $d>1$ is wide open, even in establishing the weak differentiability of $Mf$ or $\wt{M}f$ (related issues were considered by Haj\l asz and Maly in \cite{HM}). In \cite{CS}, Carneiro and Svaiter considered maximal operators of convolution type associated to smooth kernels (namely, the Gauss kernel and the Poisson kernel), and obtained the inequalities \eqref{Intro_Tanaka} and \eqref{Intro_AP} with the sharp constant $C=1$ by exploring the connections with the underlying partial differential equations. Other interesting works related to this theory are \cite{ACP, Lu1, Lu2, St}.

\smallskip
For $0 \leq \beta < d$, we define the centered fractional maximal operator as 
\begin{equation*}
M_{\beta}f(x) = \sup_{r >0} \frac{1}{m(B_r(x))^{1 - \frac{\beta}{d}}} \int_{B_r(x)} |f (y)|\,\dy.
\end{equation*}
When $\beta =0$ we plainly recover \eqref{Intro_max}. Such fractional maximal operators have applications in potential theory and partial differential equations. By comparison with an appropriate Riesz potential, one can show that if $1 < p < \infty$, $0 < \beta < d/p$ and $q = dp/(d-\beta p)$, then $M_{\beta}: L^p(\R^d) \to L^q(\R^d)$ is bounded. When $p=1$ we have again a weak-type bound (for details, see \cite[Chapter V, Theorem 1]{S}). In \cite{KiSa}, Kinnunen and Saksman studied the regularity properties of such fractional maximal operators. One of the results they proved \cite[Theorem 2.1]{KiSa} is that $M_{\beta}: W^{1,p}(\R^d) \to W^{1,q}(\R^d)$ is bounded for $p,q,\beta, d$ as described above, extending Kinunnen's original result \cite{Ki} for the case $\beta =0$. It is then natural to consider the extension of Question A to the fractional case at the endpoint $p=1$:

\smallskip

\noindent {\bf Question B.} Let $0 \leq \beta <d$ and $q = d/(d-\beta)$. Is the operator $f \mapsto |\nabla M_{\beta}f|$ bounded from $W^{1,1}(\R^d)$ to $L^q(\R^d)$? 

\smallskip

\noindent In the case $1 \leq \beta < d$, Question B admits a positive answer, which follows from the main result of Kinnunen and Saksman in their aforementioned work \cite{KiSa}. In fact, \cite[Theorem 3.1]{KiSa} states the following regularizing effect: if $f \in L^{r}(\R^d)$ with $1 < r < d$ and $1 \leq \beta < d/r$, then $M_\beta f$ is weakly differentiable and 
\begin{equation}\label{Intro_Kin_Sak}
\left| \nabla M_{\beta} f(x)\right| \leq C \,M_{\beta -1}f(x)
\end{equation}
holds for a.e. $x \in \R^d$, where $C = C(d, \beta)$ is a universal constant. In our case, given $1 \leq \beta < d$ and $f \in W^{1,1}(\R^d)$, by the Sobolev embedding we have $f \in L^{p^*}(\R^d)$, where $p^* = d/(d-1)$, and hence $f \in L^{r}(\R^d)$ for any $1 \leq r \leq p^*$. We may choose $r$ with $1< r <d$ such that $1\leq \beta < d/r$ and hence \eqref{Intro_Kin_Sak} holds. Then
\begin{equation*}
\left\| \nabla M_{\beta} f\right\|_{L^q(\R^d)} \leq C \left\| M_{\beta-1} f\right\|_{L^{q}(\R^d)} \leq C' \left\|  f\right\|_{L^{p^*}(\R^d)} \leq C'' \|\nabla f\|_{L^{1}(\R^d)}.
\end{equation*}

\noindent 

Question B (in the case $0 \leq \beta <1$) is the main motivation for this work. The presence of the fractional part introduces additional difficulties as we shall see in the course of the paper (e.g. one does not necessarily have $M_{\beta}(f) (x) \geq |f(x)|$ a.e.). Here we give a positive answer to Question B in dimension $d=1$ for the uncentered fractional maximal operator (which we denote by $\wt{M}_{\beta}$), both in the continuous and discrete settings. For general $d \geq 1$, in the discrete setting, we also obtain an interesting family of inequalities that approximate the conjectured bounds, for both the centered and uncentered versions (in a more general framework, where the balls are replaced by dilations of a given convex set). We now briefly state these results.

\subsection{Main results}
\subsubsection{Continuous setting} In dimension $d=1$, for $0\leq \beta <1$, the uncentered fractional maximal operator is 
\begin{equation}\label{Intro_frac_max_op}
\wt{M}_{\beta}f(x) = \sup_{\substack{r,s\geq 0 \\ r+s >0}} \frac{1}{(r+s)^{1 - \beta}} \int_{x-r}^{x+s} |f(y)|\,\dy.
\end{equation}
For a function $f: \R \to \R$ and $1\leq q < \infty$, motivated by the Riemann sums of a Riemann integrable function, we define its $q$-variation as
\begin{equation}\label{Intro_q_var}
\var_q(f) := \sup_{\mc{P}} \left(\sum_{n=1}^{N-1} \frac{|f(x_{n+1}) - f(x_n)|^q}{|x_{n+1} - x_n|^{q-1} }\right)^{1/q},
\end{equation}
where the supremum is taken over all finite partitions $\mc{P} = \{x_1 < x_2 < \ldots < x_N\}$. This is also known as the {\it Riesz $q$-variation} of $f$ (see, for instance, the discussion in \cite{BaLi} for this object and its generalizations). Our first result is the extension of \eqref{Intro_Tanaka} and \eqref{Intro_AP} for the uncentered fractional maximal operator, which comes with an interesting regularizing effect.
\begin{theorem}\label{Thm1}
Let $0\leq \beta < 1$ and $q = 1/(1-\beta)$. Let $f:\R \to \R$ be a function of bounded variation such that $\wt{M}_{\beta}f \not\equiv \infty$. Then $\wt{M}_{\beta}f$ is absolutely continuous and its derivative satisfies
\begin{equation}\label{Intro_q_bound1}
\big\|\big(\wt{M}_{\beta}f\big)'\big\|_{L^q(\R)} = \var_q\big(\wt{M}_{\beta}f\big) \leq 8^{1/q}\,\var(f).
\end{equation}
\end{theorem}
\noindent The constant $C=8^{1/q}$ appearing in \eqref{Intro_q_bound1} is likely not sharp. The problem of finding the sharp constant in this inequality is certainly an interesting one. Another inviting possibility is the investigation of the validity of Theorem \ref{Thm1} for the centered fractional maximal function, which, if confirmed, would be an extension of Kurka's work \cite{Ku}. It is worth mentioning that our strategy to approach the fractional case is very different from that of Tanaka \cite{Ta} and Aldaz and P\'{e}rez L\'{a}zaro \cite{AP} for the case $\beta =0$. In those papers the essential idea is to prove that  the maximal function does not have any local maxima in the set where it disconnects from the original function. In the fractional case $\beta >0$, the mere notion of the disconnecting set is ill-posed, since one does not necessarily have $M_{\beta}(f) (x) \geq |f(x)|$ a.e. anymore.

\subsubsection{Discrete setting} We shall denote  a vector $\vec{n} \in \Z^d$ by $\vec{n} = (n_1, n_2, \ldots, n_d)$. For a function $f:\mathbb{Z}^{d}\rightarrow \mathbb{R}$ (or, in general, for a vector-valued function $f:\mathbb{Z}^{d}\rightarrow \mathbb{R}^m$) we define its $\ell^{p}$-norm as usual:
\begin{equation}\label{Intro_l_p_norm}
\|f\|_{\ell^{p}{( \Z^{d})}}= \left(\sum_{\vec n\in \Z^{d}} {|f(\vec n)|^{p}}\right)^{1/p},
\end{equation}
if $1\leq p<\infty$, and
\begin{equation*}
\|f\|_{\ell^{\infty}{(\Z^{d})}}= \sup_{\vec n\in\Z^{d} }{|f(\vec n)|}.
\end{equation*}
The gradient $\nabla{f}$ of a discrete function $f:\Z^d \to \R$ is the vector
\begin{equation*}
\nabla{f(\vec n)}=\left(\frac{\partial{f}}{\partial{x_{1}}}{(\vec n)},\frac{\partial{f}}{\partial{x_{2}}}{(\vec n)},\ldots,\frac{\partial{f}}{\partial{x_{d}}}{(\vec n)} \right),
\end{equation*}
where
\begin{equation*}
\frac{\partial{f}}{\partial{x_{i}}}{(\vec n)}=f(\vec n+\vec e_{i})-f(\vec n),
\end{equation*}
and $\vec e_{i}=(0,0,\ldots,1,\ldots,0)$ is the canonical $i-$th base vector. For $f :\Z\to \R$ and $1\leq q < \infty$, the discrete analogue of \eqref{Intro_q_var} is the $q$-variation defined by
\begin{equation*}
\var_q(f) = \left(\sum_{n=-\infty}^{\infty} |f(n+1) - f(n)|^q\right)^{1/q} = \|f'\|_{\ell^q(\Z)}.
\end{equation*}

\smallskip

For $0 \leq \beta < 1$ and $f :\Z\to \R$, we define the one-dimensional discrete uncentered fractional maximal operator by
\begin{equation}\label{Intro_Def_Disc_Max_oper_dim1}
\wt{M}_{\beta}f(n) = \sup_{r,s \geq 0} \,\frac{1}{(r + s+1)^{1 - \beta}} \sum_{k = -r}^{s} |f(n +k)|.
\end{equation}
Our next result is the discrete analogue of Theorem \ref{Thm1}.

\begin{theorem}\label{Thm2}
Let $0\leq \beta < 1$ and $q = 1/(1-\beta)$. Let $f:\Z \to \R$ be a function of bounded variation such that $\wt{M}_{\beta}f \not\equiv \infty$. Then 
\begin{equation*}
\big\|\big(\wt{M}_{\beta}f\big)'\big\|_{\ell^q(\Z)} \leq 4^{1/q} \,\|f'\|_{\ell^1(\Z)}.
\end{equation*}
\end{theorem}
\noindent In the case of the discrete uncentered Hardy-Littlewood maximal function ($\beta =0$ in the setting above), Theorem \ref{Thm2} was proved by Bober, Carneiro, Hughes and Pierce in \cite{BCHP}, with the sharp constant $C=1$. The analogue of Kurka's inequality \eqref{Intro_Ku} for the discrete centered Hardy-Littlewood maximal function was established by Temur in \cite{Te} (with constant $C = 294,912,004$). As in the continuous case, the investigation of the validity of Theorem \ref{Thm2} for the discrete centered operator is also a very interesting problem.

\subsubsection{Operators associated to convex sets}\label{Subsec_Omega} We now report progress related to Question B in the multidimensional discrete setting. We do this for a more general family of fractional maximal operators defined as follows. Let $\Omega\subset \mathbb R^{d}$ be a bounded open convex set with Lipschitz boundary. Let us assume that $\vec 0\in$ int$(\Omega)$ and that $\pm \vec e_{i} \in \overline\Omega$ for $1 \leq i \leq d$ (by renormalizing $\Omega$ if necessary)\footnote{This renormalization is merely aesthetic. For general $\Omega$, one has to dilate the constants appearing in our results accordingly.}. For $r>0$ we write 
\begin{equation*}
\overline\Omega_{r}(\vec{x}_{0}) =\big\{\vec{x} \in\mathbb R^{d}; \, r^{-1}(\vec{x}-\vec{x}_{0})\in \overline{\Omega}\big\},
\end{equation*}
and for $r=0$ we consider
\begin{equation*}
\overline\Omega_{0}(\vec{x}_{0}) =\{\vec{x}_{0}\}.
\end{equation*}
Whenever $\vec{x}_{0}=\vec 0$ we shall write $\overline\Omega_{r}=\overline\Omega_{r}\big(\vec{{0}}\big) $ for simplicity. This object plays the role of the ``ball of center $\vec x_{0}$ and radius $r$" in our maximal operators below. For instance, to work with regular $\ell^{p}-$balls, one should consider $\Omega=\{\vec x\in\R^{d}; \|\vec x\|_{p}<1 \}$, where $\|\vec x\|_{p}=(|x_{1}|^{p}+|x_{2}|^{p}+\ldots+|x_{d}|^{p})^{\frac{1}{p}} $ for $\vec x=(x_{1},x_{2},\ldots,x_{d})\in \R^{d}.$

\smallskip

Given $0\leq \beta<d$ and $f:\Z^d \to \R$, we denote by $M_{\Omega, \beta}$ the discrete centered fractional maximal operator  \footnote{We remark that the results in this section remain valid if we choose $N(r)^{-1} \,r^{\beta}$ instead of $N(r)^{-1+\frac{\beta}{d}}$ in the definition of our discrete fractional maximal functions.} associated to $\Omega$, i.e. 
\begin{equation}\label{Intro_disc_Omega_cent}
M_{\Omega, \beta}f(\vec n)=\sup_{r\geq 0} \frac{1}{N(r)^{1- \frac{\beta}{d}}}\,\sum_{\vec m\in \overline\Omega_{r} }{|f(\vec n+\vec m)}|,
\end{equation}
and we denote by $\widetilde{M}_{\Omega,\beta}$ its uncentered version
\begin{equation}\label{Intro_disc_Omega_uncent}
\widetilde{M}_{\Omega,\beta}f(\vec n)=\sup_{\overline\Omega_{r}(\vec x_{0}) \owns \vec n}\, \frac{1}{N(\vec x_{0},r)^{1-\frac{\beta}{d}}}\,{\sum_{\vec m\in \overline\Omega_{r}(\vec x_{0}) }{|f( \vec m)}|},
\end{equation}
where $N(\vec x,r)$ is the number of the lattice points in the set $\overline\Omega_{r}(\vec x)$\ (and $N(r):=N(\vec 0,r)$). It should be understood throughout the rest of the paper that we always consider $\Omega$-balls with at least one lattice point.

\smallskip

These convex $\Omega$--balls have roughly the same behavior as the regular Euclidean balls from the geometric and arithmetic points of view. For instance, we have the following asymptotics \cite[Chapter VI \S 2, Theorem 2]{Lang}, for the number of lattice points
\begin{equation}\label{estimativa}
 N(\vec x,r)=C_{\Omega}\,r^{d}+O\big(r^{d-1}\big) 
\end{equation} 
as $r \rightarrow \infty$, where $C_{\Omega}=m(\Omega)$ is the $d$--dimensional volume of $\Omega,$ and the constant implicit in the big O notation depends only on the dimension $d$ and on the set $\Omega$ (e.g. if $\Omega$ is the $\ell^{\infty}$-ball we have the exact expression $N(r)=(2\lfloor r \rfloor +1)^{d}$). From \eqref{estimativa} we can find a constant $c_{1}$ depending only on the dimension $d$ and on the set $\Omega$ such that
\begin{equation}\label{Intro_Def_c1}
N(\vec x,r)\leq C_{\Omega}(r+c_{1})^{d}
\end{equation}
and
\begin{equation}\label{Intro_Def_c1_2}
N(\vec x,r)\geq \max\{C_{\Omega}(\max\{r-c_{1},0\})^{d},1\}=:C_{\Omega}(r-c_{1})^{d}_{+}.
\end{equation}
We define $c_{2}>c_{1}$ as the constant such that 
\begin{equation}\label{Intro_Def_c2}
{C_{\Omega}(c_{2}-c_{1})^{d}}=1.
\end{equation}
Since $\Omega$ is bounded, there exists $\lambda>0$ (depending only on $\Omega$) such that $\overline\Omega\subset  \overline B_{\lambda} = \overline B_{\lambda} (\vec 0)$ (note that $\lambda\geq1$, since we assume $\pm\vec {e_{i}}\in\overline\Omega$ for $1 \leq i \leq d$). This means that if $\vec p\in\overline\Omega_{r}(\vec x_{0})$ then
\begin{equation}\label{Intro_Def_lambda}
|\vec p-\vec x_{0}|\leq \lambda r.
\end{equation}

If $0< \beta<d$ and $f:\Z^d \to \R$, we consider the discrete fractional integral operator 
\begin{equation*}
I_{\beta}f(\vec n)=\sum_{\vec m\in\Z^{d}\setminus\{\vec 0\}}\frac{f(\vec n-\vec m)}{|\vec m|^{d - \beta}}.
\end{equation*}
One can show that 
$$M_{\Omega,\beta} f(\vec n)\leq C\,I_{\beta}|f|(\vec n)+|f(\vec n)|$$ 
for all $ \vec n\in \Z^{d}$, where $C = C(d, \Omega,\beta)$. It is known that if $1 < p < \infty$, $0 < \beta < d/p$ and $q = dp/(d-\beta p)$, then $I_{\beta}:\ell^{p}(\Z^{d})\rightarrow \ell^{q}(\Z^{d})$ is bounded (see, for instance, Pierce's thesis \cite[Proposition 2.4]{P}). Observe also that, if $p \leq q$, then 
\begin{equation}\label{Intro_inclusion}
\|f\|_{\ell^{q}(\Z^d)} \leq \|f\|_{\ell^{p}(\Z^d)}.
\end{equation}
This plainly implies that $M_{\Omega,\beta}:\ell^{p}(\Z^{d})\rightarrow \ell^{q}(\Z^{d})$ is bounded for $p,q,\beta,d$ as above. One can also verify the pointwise inequality
\begin{equation*}
\left|\frac{\partial}{\partial{x_{i}}}(M_{\Omega,\beta} f)(\vec n)\right|\leq M_{\Omega,\beta}\left(\frac{\partial{f}}{\partial{x_{i}}}\right)(\vec n)
\end{equation*}
for all $\vec n \in \Z^d$. Hence, for $f \in \ell^p(\Z^d)$, we have \footnote{Throughout this paper our constants may vary from line to line.} 
\begin{equation}\label{Intro_disc_bound}
\|\nabla M_{\Omega,\beta}f\|_{\ell^{q}(\Z^{d})}\leq C\,\| M_{\Omega,\beta}|\nabla f|\|_{\ell^{q}(\Z^{d})}  \leq C' \,\|\nabla f\|_{\ell^{p}(\Z^{d})}.\end{equation}
Moreover, if $f,g\in \ell^{p}(\Z^{d})$, we have (recall that derivation is a bounded operator in $\ell^{q}(\Z^d)$, by the triangle inequality)
\begin{equation*}
\|\nabla M_{\Omega,\beta}f-\nabla M_{\Omega,\beta}g\|_{\ell^{q}(\Z^{d})}\leq C\, \| M_{\Omega,\beta}f-M_{\Omega,\beta}g\|_{\ell^{q}(\Z^{d})}\leq C\, \|M_{\Omega, \beta}{|f-g|} \|_{\ell^{q}(\Z^{d})} \leq C'\,\| f-g\|_{\ell^{p}(\Z^{d})},
\end{equation*}
and we see that the operator $f\mapsto \nabla M_{\Omega,\beta}f$ is continuous from $\ell^{p}(\Z^{d})$ to $\ell^{q}(\Z^{d})$. Similar remarks apply to the uncentered version $\wt{M}_{\Omega,\beta}$. In relation to \eqref{Intro_disc_bound}, the conjectured bound suggested by Question B in the discrete endpoint case $p=1$ is the following:

\smallskip

\noindent {\bf Question C.} Let $0 \leq \beta <d$ and $q = d/(d-\beta)$. For a discrete function $f:\Z^d \to \R$ do we have $\|\nabla M_{\Omega,\beta}f\|_{\ell^{q}(\Z^{d})}\leq C (d, \Omega,\beta) \,\|\nabla f\|_{\ell^{1}(\Z^{d})}$?

\smallskip

\noindent Our next result is related to this question. In the case $0 \leq \beta <1$ we present a family of estimates that approximate the conjectured bounds, whereas in the case $1\leq \beta <d$ we give a positive answer (under the assumption that $f \in \ell^1(\Z^d)$) by adapting the methods of Kinnunen and Saksman \cite{KiSa} to the discrete setting. We complement these results with their corresponding continuity statements.

\begin{theorem}\label{Thm3}.
\begin{enumerate}
\item[(i)] Let $0\leq \beta <d$ and $0 \leq \alpha \leq 1$. Let $q \geq 1$ be such that 
\begin{equation}\label{Thm3_condition}
q > \frac{d}{d - \beta + \alpha}.
\end{equation}
Then there exists a constant $C = C(d,\Omega,\alpha,\beta,q) >0$ such that
\begin{equation}\label{main ineq cent}
\|\nabla M_{\Omega,\beta}f\|_{\ell^q(\Z^d)} \leq C\,\|\nabla f\|_{\ell^1(\Z^d)}^{1-\alpha}\,\|f\|^{\alpha}_{\ell^{1}(\Z^{d})}\ \ \forall f\in \ell^{1}(\Z^{d}).
\end{equation}
Moreover, the operator $f\mapsto \nabla M_{\Omega,\beta}f$ is continuous from $\ell^{1}(\Z^d)$ to $\ell^{q}(\Z^{d})$.

\smallskip

\item[(ii)] Let $1\leq \beta <d$ and $0 \leq \alpha < 1$. Let 
\begin{equation}\label{Thm3_condition_part2}
q = \frac{d}{d - \beta + \alpha}.
\end{equation}
Then there exists a constant $C = C(d,\Omega,\alpha,\beta) >0$ such that
\begin{equation}\label{main ineq cent_part2}
\|\nabla M_{\Omega,\beta}f\|_{\ell^q(\Z^d)} \leq C\,\|\nabla f\|_{\ell^1(\Z^d)}^{1-\alpha}\,\|f\|^{\alpha}_{\ell^{1}(\Z^{d})}\ \ \forall f\in \ell^{1}(\Z^{d}).
\end{equation}
Moreover, the operator $f\mapsto \nabla M_{\Omega,\beta}f$ is continuous from $\ell^{1}(\Z^d)$ to $\ell^{q}(\Z^{d})$.
\end{enumerate} 

\smallskip

The same results hold for the discrete uncentered fractional maximal operator $\widetilde{M}_{\Omega,\beta}$. 
\end{theorem}

Theorem \ref{Thm3} extends the result of Carneiro and Hughes in \cite[Theorem 1]{CH}, which corresponds to the case $\beta = 0$, $\alpha = 1$ and $q=1$. Our approach here is different and simpler than that of \cite{CH}. The boundedness part of Theorem \ref{Thm3} (i) for the classical discrete Hardy-Littlewood maximal operator ($\beta = 0$) and $q=1$ was first established by Carneiro and Rogers (unpublished manuscript). It is important to observe that inequality \eqref{main ineq cent} (and its analogue for the uncentered case) can only hold if 
\begin{equation}\label{Intro_nec_cond}
q \geq \frac{d}{d - \beta + \alpha}.
\end{equation}
This is due, essentially, to a dilation argument. To see this, let us consider, for instance, the uncentered case where $\Omega = (-1,1)^d$ is the unit open cube. Let $k \in \N$ and consider the cube $Q_k = [-k,k]^d$ and its characteristic function $f_k := \chi_{Q_k}$. One has $\|f_k\|_{\ell^{1}(\Z^d)} \sim_d k^d$, $\|\nabla f_k\|_{\ell^1(\Z^d)} \sim_d k^{d-1}$ and $\|\nabla \wt{M}_{\Omega,\beta}f_k\|_{\ell^q(\Z^d)} \gg_{\Omega,\beta,d} k^{\frac{d}{q}-1 + \beta}$. One can see this last estimate by considering the region $H = \{\vec{n} = (n_1, n_2, \ldots, n_d) \in \Z^d; \,n_1 \geq 4dk\,; \ |n_i| \leq k,\, {\rm for}\,\, i = 2,3,\ldots, d\}$ and showing that the maximal function at $\vec{n} \in H$ is realized by the cube of side $n_1 + k$ that contains the cube $Q_k$. Then we sum $|\wt{M}_{\Omega,\beta}f_k (\vec n + \vec{e}_1) - \wt{M}_{\Omega,\beta}f_k (\vec n)|^q$ from $n_1 = 4dk$ to $\infty$, and then sum these contributions over the $\sim k^{d-1}$ possibilities for $(n_2, \ldots, n_d)$. Letting $k \to \infty$ we obtain the necessary condition \eqref{Intro_nec_cond}. 

\smallskip

We may collect the cases left open in Theorem \ref{Thm3} in our final question:

\smallskip

\noindent {\bf Question D.} Does the inequality \eqref{main ineq cent} (and its analogue for the uncentered case) hold for all $\alpha \leq \beta$ and $q = d/(d - \beta + \alpha)$?  

\smallskip

We now proceed to the proofs of these results. In doing so, we opt to consider the discrete cases first, since they describe the essence of the main ideas with a little less technicalities than the continuous cases. In Section \ref{Sec2} we prove the boundedness part of Theorem \ref{Thm3}. In Section \ref{Sec3} we prove the continuity part of Theorem \ref{Thm3}, a nontrivial statement that does not follow directly from the boundedness, as the maximal operators are no longer sublinear at the derivative level. In Section \ref{Sec4} we prove Theorem \ref{Thm2} and, finally, in Section \ref{Sec5} we adapt some of our ideas used in the discrete setting to the continuous setting, and conclude by proving Theorem \ref{Thm1}.

\section{Proof of Theorem \ref{Thm3} - Boundedness}\label{Sec2}

Throughout this section we work with the discrete maximal operators \eqref{Intro_disc_Omega_cent} and \eqref{Intro_disc_Omega_uncent} associated to a convex set $\Omega$ as described in \S \ref{Subsec_Omega}, and we remove the subscript $\Omega$ in some passages for simplicity. Given such a convex set $\Omega$, let us fix the constants $C_\Omega = m(\Omega)$, $c_1$, $c_2$ and $\lambda$ as defined in \eqref{estimativa} - \eqref{Intro_Def_lambda}. In proving the boundedness statements of Theorem \ref{Thm3} we may assume, without loss of generality, that $f$ is nonnegative since $\big|\nabla |f|(\vec{n})\big| \leq  \big|\nabla f(\vec{n})\big|$ and $M _{\beta}|f| = M_{\beta}f$ (resp. $\wt{M} _{\beta}|f| = \wt{M}_{\beta}f$).

\subsection{Centered case - part (i)} \label{Cent_case_part_i} To prove \eqref{main ineq cent} it is sufficient to show that
\begin{equation}\label{object d}
\left\| \frac{\partial{M_{\beta}f}}{\partial x_{i}}\right\|^{q}_{\ell^{q}{(\Z^d)}}\leq C \,\|\nabla f\|_{\ell^1(\Z^d)}^{(1-\alpha)q}\,{\|f\|^{\alpha q}_{\ell^{1}{(\Z^{d})}}} 
\end{equation}
for each $i=1,2,\dots,d.$ We will work with $i=d$ and the other cases are analogous. Since $f\in \ell^{1}{(\Z^{d})}$, for each $\vec n\in \mathbb{Z}^d$ there exists $r\geq 0$ such that
\begin{equation}\label{Sec2_eq_average}
M_{\beta}f(\vec n)=A_{r}f(\vec n):=\frac{1}{N(r)^{1-\frac{\beta}{d}}}{\sum_{\vec m\in \Omega_{r}}{f(\vec n+\vec m )}}.
\end{equation}
Let $\ell(\vec n)$  be the minimum $\ell \in \Z^+$ such that there exists $r \geq 0$  which satisfies \eqref{Sec2_eq_average} and $\lceil r\rceil=\ell(\vec n)$. Also, let $r(\vec n) \geq 0$ be such that $r(\vec n)$ satisfies \eqref{Sec2_eq_average} and $\lceil r(\vec n)\rceil=\ell(\vec n).$

\smallskip

For all $k\in \Z^+$ we define the set $X^{+}_{k}$ by
\begin{equation}\label{Def_X_n_+}
X^{+}_{k}:=\{\vec n \in\Z^d: M_{\beta}f(\vec n + \vec{e}_d)\geq M_{\beta}f(\vec n) \ {\rm and}\ {\ell(\vec n + \vec{e}_d)}=k \}
\end{equation}
and the set $X^{-}_{k}$ by
\begin{equation}\label{Def_X_n_-}
X^{-}_{k}:=\{\vec n \in\Z^d: M_{\beta}f(\vec n + \vec{e}_d)< M_{\beta}f(\vec n) \ {\rm and} \  \ell(\vec n) =k \}.
\end{equation}
Hence, we may write 
\begin{align*}
\sum_{\vec n \in \Z^d}{\left|\frac{\partial}{\partial x_{d}}{M_{\beta}f(\vec n)}\right|}^{q}&=\sum_{k\geq0}\, \sum_{\vec n \in X^{+}_{k}}\big(M_{\beta}f(\vec n + \vec{e}_d)-M_{\beta}f(\vec n)\big)^{q} +  \sum_{k\geq0}\, \sum_{\vec n \in X^{-}_{k}}\big(M_{\beta}f(\vec n)-M_{\beta}f(\vec n + \vec{e}_d)\big)^{q}.
\end{align*}
We will prove that
\begin{equation}\label{object -}
\sum_{k\geq0}\, \sum_{\vec n \in X^{-}_{k}}\big(M_{\beta}f(\vec n)-M_{\beta}f(\vec n + \vec{e}_d)\big)^{q}\leq C\,\|\nabla f\|_{\ell^1(\Z^d)}^{q(1-\alpha)}\,{\|f\|^{q\alpha}_{\ell^{1}(\Z^{d})}},
\end{equation}
and, analogously, we will have
\begin{equation}\label{object +}
\sum_{k\geq0}\, \sum_{\vec n \in X^{+}_{k}}\big(M_{\beta}f(\vec n + \vec{e}_d)-M_{\beta}f(\vec n)\big)^{q} \leq C\,\|\nabla f\|_{\ell^1(\Z^d)}^{q(1-\alpha)}\,{\|f\|^{q\alpha}_{\ell^{1}(\Z^{d})}}.
\end{equation}
Inequalities \eqref{object -} and \eqref{object +} then imply \eqref{object d} for $i=d$ \footnote{Recall that our constants may vary from line to line.}. To show \eqref{object -}, first we note that for $\vec n \in X^{-}_{k}$ and $r = r(\vec n)$ we have 
\begin{align}\label{ineq fund 1}
\begin{split}
M_{\beta}f(\vec n)-M_{\beta}f(\vec n + \vec{e}_d)&\leq A_{r}{f(\vec n)}-A_{r}{f(\vec n + \vec{e}_d)}\\
&= \frac{1}{N( r)^{1-\frac{\beta}{d}}}{\sum_{\vec m\in \overline\Omega_{r}}{f(\vec n+\vec m )}}-\frac{1}{N( r)^{1-\frac{\beta}{d}}}{\sum_{\vec m\in \overline\Omega_{r}}{f(\vec n + \vec{e}_d +\vec m )}}\\
&\leq\frac{1}{N^{+}(k-1)^{1-\frac{\beta}{d}}}{\sum_{\vec m\in \overline\Omega_{k}}{\big|f(\vec n+\vec m )-f(\vec n + \vec{e}_d +\vec m)\big|}}\,,
\end{split}
\end{align}
where $N^{+}(\vec x,\tau):=\max\{N(\vec x,\tau),1\}$, and $N^{+}( \tau)=N^{+}(\vec0,\tau)$ \footnote{Here we formally include the possibility of having $\tau <0$, with the understanding that, in this case, $N^{+}(\vec x,\tau) = 1$.}. On the other hand,
\begin{align}\label{Sec1_eq_conv_e_d}
\begin{split}
M_{\beta}f(\vec n + \vec{e}_d)&\geq\frac{1}{N(k+1)^{1-\frac{\beta}{d}}}{\sum_{\vec m\in \overline\Omega_{   k+1}}{f(\vec n + \vec{e}_d +\vec m )}}\\
&\geq \frac{N(k+1)^{\frac{\beta}{d}}}{N(k+1)}{\sum_{\vec m\in \overline\Omega_{k}}{f(\vec n+\vec m )}}\\
&\geq\frac{N(r)}{N(k+1)}{M_{\beta}f(\vec n)}.
\end{split}
\end{align}
In the second inequality above we have used the convexity of $\Omega$ and the fact that $\vec {-e_{d}}\in\overline\Omega$ to conclude that $\overline{\Omega}_{k}(\vec n) \subset \overline{\Omega}_{k+1}(\vec n + \vec{e}_d)$. Using \eqref{Sec1_eq_conv_e_d} we get
\begin{align}\label{ineq fund 2}
\begin{split}
M_{\beta}f(\vec n)-M_{\beta}f(\vec n + \vec{e}_d)&\leq \left(1-\frac{N(r)}{N(k+1)} \right)M_{\beta}f(\vec n)\\
&\leq \frac{N(k+1)-N^{+}(k-1)}{N(k+1)}\frac{1}{N^{+}(k-1)^{1-\frac{\beta}{d}}}\sum_{\vec m\in \overline\Omega_{k}}{f(\vec n+\vec m)}.
\end{split}
\end{align}
Putting together the estimates \eqref{ineq fund 1} and \eqref{ineq fund 2}, we see that $\big(M_{\beta}f(\vec n)-M_{\beta}f(\vec n + \vec{e}_d)\big)^{q}$ is  bounded above by the product \footnote{If $\alpha = 0\,\,{\rm or}\,1$, it is understood that we only have one term in this product. The modifications for the rest of the proof are standard.}
\begin{align}\label{cota sup 1}
\begin{split}
& \left( \frac{1}{N^{+}(k-1)^{1-\frac{\beta}{d}}}{\sum_{\vec m\in \overline\Omega_{k}}{\big|f(\vec n +\vec m )-f(\vec n + \vec{e}_d+\vec m)\big|}}\right)^{q(1-\alpha)} \\
& \ \ \ \ \ \ \ \ \ \ \ \ \ \ \ \ \  \ \ \ \ \ \ \times \left( \frac{N(k+1)-N^{+}(k-1)}{N(k+1)}\frac{1}{N^{+}(k-1)^{1-\frac{\beta}{d}}}\sum_{\vec m\in \overline\Omega_{k}}{f(\vec n + \vec m)}\right)^{q\alpha}.
\end{split}
\end{align}

By \eqref{Intro_Def_c1} and \eqref{Intro_Def_c1_2}, there is a positive constant $C$ such that
\begin{equation}\label{des N}
\frac{N(\vec x,k+1)-N^{+}(\vec x,k-1)}{N(\vec x,k+1)}\leq \frac{C}{N^{+}(\vec x,k-1)^{\frac{1}{d} }}\ \ \forall k\in\Z^+  \,;\,  \forall \vec x \in \R^d.
\end{equation}
Let 
\begin{equation}\label{Def_gamma}
\left( 1 - \frac{\beta}{d}\right)q + \frac{\alpha q}{d} = 1 + \gamma.
\end{equation}
From \eqref{Thm3_condition} we have $\gamma >0$. Then it follows from \eqref{cota sup 1} and \eqref{des N} that $\big(M_{\beta}f(\vec n)-M_{\beta}f(\vec n + \vec{e}_d)\big)^{q} $ is bounded above by the product 
\begin{align*}
& \left( \frac{1}{N^{+}(k-1)^{1+\gamma}}{\left(\sum_{\vec m\in \overline\Omega_{k}}{\big|f(\vec n+\vec m )-f(\vec n + \vec{e}_d+\vec m)\big|}\right)^{q}}\right)^{(1-\alpha)} \!\!\!\!\!\! \times \left( \frac{C}{N^{+}(k-1)^{1+ \gamma}}\left(\sum_{\vec m\in \overline\Omega_{k}}{f(\vec n+\vec m)}\right)^{q}\right)^{\alpha}.
\end{align*}
Thus, by H\"{o}lder's inequality with exponents $p=\frac{1}{1-\alpha}$ and $p'=\frac{1}{\alpha}$, we see that the left-hand side of \eqref{object -} is bounded by 
\begin{align*}
\begin{split}
& \left[ \sum_{k\geq0}\, \sum_{\vec n \in X^{-}_{k}} \frac{1}{N^{+}(k-1)^{1+ \gamma}}\left({\sum_{\vec m\in \overline\Omega_{k}}{\big|f(\vec n+\vec m )-f(\vec n + \vec{e}_d +\vec m)\big|}}\right)^{q}\right]^{1-\alpha}\\
& \ \ \ \ \ \ \ \ \ \ \ \ \ \ \ \ \  \ \ \ \ \ \ \ \ \ \ \ \ \  \times \left[ \sum_{k\geq0}\, \sum_{\vec n \in X^{-}_{k}} \frac{C}{N^{+}(k-1)^{1+\gamma}}\left(\sum_{\vec m\in \overline\Omega_{k}}{f(\vec n+\vec m)}\right)^{q}\right]^{\alpha}.
\end{split}
\end{align*}
Since $q \geq 1$, this last product is bounded by
\begin{align*}
\begin{split}
& \|\nabla f\|_{\ell^1(\Z^d)}^{(q-1)(1-\alpha)} \left[\sum_{k\geq0}\, \sum_{\vec n \in X^{-}_{k}} \frac{1}{N^{+}(k-1)^{1+\gamma}}\left({\sum_{\vec m\in \overline\Omega_{k}}{\big|f(\vec n +\vec m )-f(\vec n + \vec{e}_d+\vec m)\big|}}\right)\right]^{1-\alpha}\\
& \ \ \ \ \ \ \ \ \ \ \ \ \ \ \ \ \  \ \ \ \ \ \ \times \| f\|_{\ell^1(\Z^d)}^{(q-1)\alpha} \left[ \sum_{k\geq0}\, \sum_{\vec n \in X^{-}_{k}} \frac{C}{N^{+}(k-1)^{1+\gamma}}\left(\sum_{\vec m\in \overline\Omega_{k}}{f(\vec n+\vec m)}\right)\right]^{\alpha}.
\end{split}
\end{align*}
By Fubini's theorem and \eqref{Intro_Def_lambda}, this is bounded by 
\begin{align*}
\begin{split}
& \|\nabla f\|_{\ell^1(\Z^d)}^{(q-1)(1-\alpha)} \left[ \sum_{\vec m\in \Z^d} \sum_{k \geq \frac{|\vec m|}{\lambda}} \frac{1}{N^{+}(k-1)^{1+\gamma}}\sum_{\vec n \in X^{-}_{k}} \big|f(\vec n+\vec m )-f(\vec n + \vec{e}_d +\vec m)\big| \right]^{1-\alpha}\\
& \ \ \ \ \ \ \ \ \ \ \ \ \ \ \ \ \  \ \ \ \ \ \ \times \| f\|_{\ell^1(\Z^d)}^{(q-1)\alpha} \left[ \sum_{\vec m\in \Z^d} \sum_{k \geq \frac{|\vec m|}{\lambda}} \frac{C}{N^{+}(k-1)^{1+\gamma}} \sum_{\vec n \in X^{-}_{k}} f(\vec n+\vec m)\right]^{\alpha},
\end{split}
\end{align*}
which turns out to be bounded by 
\begin{align}\label{Sec2_Important_estimate}
\begin{split}
& \|\nabla f\|_{\ell^1(\Z^d)}^{(q-1)(1-\alpha)} \left[ \sum_{\vec m\in \Z^d}  \frac{1}{N^{+}\big(\frac{|\vec{m}|}{\lambda}-1\big)^{1+\gamma}}\sum_{k \geq \frac{|\vec m|}{\lambda}}\sum_{\vec n \in X^{-}_{k}} \big|f(\vec n+\vec m )-f(\vec n + \vec{e}_d +\vec m)\big| \right]^{1-\alpha}\\
& \ \ \ \ \ \ \ \ \ \ \ \ \ \ \ \ \  \ \ \ \ \ \ \times \| f\|_{\ell^1(\Z^d)}^{(q-1)\alpha} \left[ \sum_{\vec m\in \Z^d} \frac{C}{N^{+}\big( \frac{|\vec m|}{\lambda} -1\big)^{1+\gamma}} \sum_{k \geq \frac{|\vec m|}{\lambda}} \sum_{\vec n \in X^{-}_{k}}f(\vec n+\vec m)\right]^{\alpha}.
\end{split}
\end{align}
Since the sets $X^{-}_{k}$ are pairwise disjoint, this is bounded above by
\begin{align}\label{ineq holder_with_q-1_v4}
\begin{split}
\|\nabla f\|_{\ell^1(\Z^d)}^{q(1-\alpha)} \left[ \sum_{\vec m\in \Z^d}  \frac{1}{N^{+}\big(\frac{|\vec{m}|}{\lambda}-1\big)^{1+\gamma}} \right]^{1-\alpha} \times \ \  \| f\|_{\ell^1(\Z^d)}^{q\alpha} \left[ \sum_{\vec m\in \Z^d} \frac{C}{N^{+}\big( \frac{|\vec m|}{\lambda} -1\big)^{1+\gamma}} \right]^{\alpha}.
\end{split}
\end{align}

By \eqref{Intro_Def_c1_2} we know that
\begin{equation*}
N^+\left(\frac{|\vec m|}{\lambda}-1\right)\geq \max\left\{C_{\Omega}\left(\max\left\{\frac{|\vec m|}{\lambda}-1-c_{1},0\right\}\right)^{d},1\right\} \ \ \forall \, \vec m\in\mathbb Z^{d}.
\end{equation*}
This implies that both sums in brackets in \eqref{ineq holder_with_q-1_v4} are finite and we arrive at the desired estimate \eqref{object -}. Using \eqref{Intro_Def_c1} we see that 
$$\sum_{\vec m\in \mathbb Z^{d}} \frac{1}{N^+\Big(\frac{|\vec{m}|}{\lambda}-1\Big)^{1+\gamma}}\to \infty\ \ \text{as}\ \ \gamma \to 0^+\,,$$
which implies that our constant $C$ in \eqref{main ineq cent} blows up when we approach the case of equality in \eqref{Intro_nec_cond}. 


\subsection{Uncentered case - part (i)} Here it suffices to show that 
\begin{equation}\label{object noncent d}
\left\| \frac{\partial{\widetilde M_{\beta} f}}{\partial x_{i}}\right\|^{q}_{\ell^{q}(\Z^{d})}\leq \,C \,\|\nabla f\|_{\ell^1(\Z^d)}^{(1-\alpha)q}\,\|f\|^{\alpha q}_{\ell^{1}(\Z^{d})}
\end{equation}
for each $i=1,2,\dots,d$. We work again with $i=d$ and the other cases are analogous. Since $f\in \ell^{1}{(\mathbb{Z}^{d})}$,  for each $\vec n \in \Z^{d}$ we can take a point $\vec x=\vec x(\vec n) \in \R^{d}$ and a radius $r = r(\vec n)$ such that $\vec n \in\overline \Omega_{r}(\vec x)$ and the fractional average over the set $\overline \Omega_{r}(\vec x)$ realizes the supremum in the maximal function, i.e.
\begin{equation}\label{Sec3_max-ave}
\widetilde M_{\beta} f(\vec n)=A_{(\vec x,r)}f(\vec n):=\frac{1}{N(\vec x,r)^{1-\frac{\beta}{d}}}{\sum_{\vec m\in \overline\Omega_{r}
(\vec x)}{f(\vec m )}}.
\end{equation}
\noindent
Let $\tilde{\ell}(\vec n)$ be the minimum $\tilde{\ell} \in\Z^+$ such that there is a pair $(\vec{x}, r)$ that verifies the equality \eqref{Sec3_max-ave} with $\lceil r\rceil = \tilde{\ell}(\vec n)$. Let $r(\vec n)$ be such a radius and $\vec{x}(\vec n)$ be such a center.

\smallskip

For $k \in \Z^+$ we now define the set $\widetilde{X}^{+}_{k}$ by
\begin{equation*}
\widetilde{X}^{+}_{k}:=\big\{\vec n\in\Z^d: \widetilde M_{\beta}f(\vec n + \vec{e}_d)\geq \widetilde M_{\beta}f(\vec n) \ {\rm and}\ {\tilde{\ell}(\vec n +\vec{e}_d)}=k \big\}
\end{equation*}
and the set $\widetilde X^{-}_{k}$ by
\begin{equation*}
\widetilde X^{-}_{k}:=\big\{\vec n \in\Z^d: \widetilde M_{\beta}f(\vec n + \vec{e}_d)< \widetilde M_{\beta}f(\vec n) \ {\rm and} \  \tilde{\ell}(\vec n) =k \big\}.
\end{equation*}
Hence,
\begin{align*}
\sum_{\vec n \in \Z^d}{\left|\frac{\partial}{\partial x_{d}}{\widetilde M_{\beta} f(\vec n)}\right|^{q}}&= \, \sum_{k\geq0}\sum_{\vec n \in X^{+}_{k}}{\big(\widetilde M_{\beta} f(\vec n + \vec{e}_d)-\widetilde M_{\beta} f(\vec n)}\big)^{q} +\sum_{k\geq0}\sum_{\vec n \in X^{-}_{k}}{\big(\widetilde M_{\beta} f(\vec n)-\widetilde M_{\beta} f(\vec n + \vec{e}_d)}\big)^{q}.
\end{align*}
We will prove that
\begin{equation}\label{object noncent d -}
\sum_{k\geq0}\sum_{\vec n \in X^{-}_{k}}{\big(\widetilde M_{\beta} f(\vec n)-\widetilde M_{\beta} f(\vec n + \vec{e}_d)}\big)^{q}\leq C\,\|\nabla f\|_{\ell^1(\Z^d)}^{q(1-\alpha)}\,{\|f\|^{q\alpha}_{\ell^{1}(\Z^{d})}},
\end{equation}
and, analogously, we will have 
\begin{equation}\label{object noncent d +}
\sum_{k\geq0}\sum_{\vec n \in X^{+}_{k}}{\big(\widetilde M_{\beta} f(\vec n + \vec{e}_d)-\widetilde M_{\beta} f(\vec n)}\big)^{q} \leq C\,\|\nabla f\|_{\ell^1(\Z^d)}^{q(1-\alpha)}\,{\|f\|^{q\alpha}_{\ell^{1}(\Z^{d})}}.
\end{equation}
As a consequence of \eqref{object noncent d -} and \eqref{object noncent d +} we will obtain \eqref{object noncent d} for $i=d$, as desired.

\smallskip

For $\vec n \in X^{-}_{k}$ let $r = r(\vec n)$ and $\vec x= \vec{x}(\vec n)$. Then 
\begin{align*}
\widetilde M_{\beta} f(\vec n)-\widetilde M_{\beta} f(\vec n + \vec{e}_d)&\leq A_{(\vec x ,r)}{f(\vec n)}-A_{(\vec x +\vec e_{d},r)}{f(\vec n + \vec{e}_d)}\\
&= \frac{1}{N(\vec x ,r)^{1-\frac{\beta}{d}}}{\sum_{\vec m\in \Omega_{r}{(\vec x)}}{f(\vec m )}} -\frac{1}{N(\vec x +\vec e_{d},r)^{1-\frac{\beta}{d}}}{\sum_{\vec m\in \Omega_{r}{(\vec x +\vec e_{d})}}{f(\vec m )}}\\
&\leq \frac{1}{N(\vec x,r)^{1-\frac{\beta}{d}}}\sum_{\vec m\in \Omega_{r}(\vec x)}{\big|f(\vec m)-f(\vec m +\vec e_{d})\big|}\\
&\leq \frac{1}{N^{+}(\vec x,k-1)^{1-\frac{\beta}{d}}}\sum_{\vec m\in \Omega_{k}(\vec x)}{\big|f(\vec m)-f(\vec m +\vec e_{d})\big|}.
\end{align*}
On the other hand,
\begin{align*}
\widetilde M_{\beta} f(\vec n + \vec{e}_d)&\geq\frac{1}{N(\vec x + \vec e_d,k+1)^{1-\frac{\beta}{d}}}{\sum_{\vec m\in \Omega_{k+1}{ (\vec x + \vec e_d)}}{f(\vec m )}}\\
&\geq \frac{N(\vec x ,k+1)^{\frac{\beta}{d}}}{N(\vec x ,k+1)}{\sum_{\vec m\in \Omega_{k}{(\vec x)}}{f(\vec m )}}\\
&\geq\frac{N(\vec x,r)}{N(\vec x ,k+1)}\,\widetilde M_{\beta} f(\vec n).
\end{align*}
In the second inequality above we have used the convexity of $\Omega$ and the fact that $\vec {-e_{d}}\in\overline\Omega$ to conclude that $\overline{\Omega}_{k}(\vec x) \subset \overline{\Omega}_{k+1}(\vec x + \vec{e}_d)$. Therefore,
\begin{align*}
\widetilde M_{\beta} f(\vec n)-\widetilde M_{\beta} f(\vec n + \vec{e}_d)&\leq \,\left(1-\frac{N(\vec x ,r)}{N(\vec x ,k+1)} \right)\widetilde M_{\beta} f(\vec n)\\
&\leq\, \frac{N(\vec x,k+1)-N^+(\vec x,k-1)}{N(\vec x,k+1)}\frac{1}{N^+(\vec x,k-1)^{1-\frac{\beta}{d}}}\sum_{\vec m\in \overline\Omega_{k}{(\vec x)}}{f(\vec m)}.
\end{align*}
With these two inequalities, we can see that $\big(\widetilde M_{\beta} f(\vec n)-\widetilde M_{\beta} f(\vec n + \vec{e}_d)\big)^{q} $ is bounded above by the product 
\begin{align}\label{alfa 1 non cent}
\begin{split}
& \left( \frac{1}{N^{+}(\vec x,k-1)^{1-\frac{\beta}{d}}}\sum_{\vec{m}\in \Omega_{k}(\vec x)}{\big|f(\vec{m})-f(\vec m +\vec e_{d})\big|}\right)^{q(1-\alpha)}\\
&  \ \ \ \ \ \ \ \ \ \ \ \ \ \ \ \ \ \ \times \left( \frac{N(\vec x,k+1)-N^{+}(\vec x,k-1)}{N(\vec x,k+1)}\frac{1}{N^{+}(\vec x,k-1)^{1-\frac{\beta}{d}}}\sum_{\vec m\in \overline\Omega_{k}{(\vec x)}}{f(\vec m)}\right)^{q\alpha}.
\end{split}
\end{align}
Let $\gamma>0$ be defined as in \eqref{Def_gamma}. Using \eqref{des N}, we conclude that \eqref{alfa 1 non cent} is bounded above by the product 
\begin{align*}
& \left( \frac{1}{N^{+}(\vec x,k-1)^{1+\gamma}}\left({\sum_{\vec m\in \Omega_{k}{(\vec x)}}{\big|f(\vec m )-f(\vec m +\vec e_{d})\big|}}\right)^{q}\right)^{1-\alpha} \!\! \times \ \left( \frac{C}{N^{+}(\vec x,k-1)^{1+\gamma}}\left(\sum_{\vec m\in \overline\Omega_{k}{(\vec x)}}{f(\vec m)}\right)^{q}\right)^{\alpha}.
\end{align*}
For each $r\in \R$, let $\widetilde N^+(r):=\min\big\{N^+(\vec x,r); \, \vec x\in\mathbb R^{d} \big\}$. Thus, by H\"{o}lder's inequality with exponents $p=\frac{1}{1-\alpha}$ and $p'=\frac{1}{\alpha}$ we see that the left-hand side of \eqref{object noncent d -} is bounded by
\begin{align*}
& \left[ \sum_{k\geq0}\sum_{\vec n \in X^{-}_{k}} \frac{1}{\widetilde N^+(k-1)^{1+\gamma}}\left({\sum_{\vec m\in \Omega_{k}(\vec x (\vec n))}{\big|f(\vec m )-f(\vec m +\vec e_{d})\big|}}\right)^{q}\right]^{1-\alpha}\\
&  \ \ \ \ \ \ \ \ \ \ \ \ \ \ \ \ \ \  \ \ \ \ \ \ \ \ \ \times  \left[\sum_{k\geq0}\sum_{\vec n \in X^{-}_{k}}\frac{C}{\widetilde N^+(k-1)^{1+\gamma}}\left(\sum_{\vec m\in \overline\Omega_{k}(\vec x(\vec n))}{f(\vec m)}\right)^{q}\right]^{\alpha}.
\end{align*}
In turn, this is bounded by 
\begin{align*}
&\left[ \sum_{k\geq0}\sum_{\vec n \in X^{-}_{k}} \frac{1}{\widetilde N^+(k-1)^{1+\gamma}}\left({\sum_{\vec m\in \overline B_{2 k \lambda}(\vec n)}{\big|f(\vec m )-f(\vec m +\vec e_{d})\big|}}\right)^{q}\right]^{1-\alpha}\\
&  \ \ \ \ \ \ \ \ \ \ \ \ \ \ \ \ \ \  \ \ \ \ \ \ \ \ \ \times \left[ \sum_{k\geq0}\sum_{\vec n \in X^{-}_{k}}\frac{C}{\widetilde N^+(k-1)^{1+\gamma}}\left(\sum_{\vec m\in \overline B_{2k\lambda}(\vec n)}{f(\vec m)}\right)^{q}\right]^{\alpha},
\end{align*}
since $\overline{\Omega}\subset \overline{B}_{\lambda}$ and thus $\overline{\Omega}_{k}(\vec x(\vec n))\subset \overline{B}_{k\lambda}(\vec x(\vec n)) \subset \overline{B}_{2k\lambda}(\vec n)$. The remaining steps of the proof are analogous to the centered case in \S \ref{Cent_case_part_i}.

\subsection{Centered case - part (ii)} \label{Cent_case_endpoint} 
In the case $1 \leq \beta < d$, the operator $M_{\beta}$ has a certain regularizing effect. This was observed in \cite[Theorem 3.1]{KiSa} in the continuous setting. In what follows we adapt their argument to the discrete setting. Let $\vec{n} \in \Z^d$ and assume that $M_{\beta} f(\vec{n}) \geq M_{\beta} f(\vec{n} + \vec{e}_d)$. Since we are assuming that $f \in \ell^1(\Z^d)$, there exists $r\geq 0$ such that
\begin{equation*}
M_{\beta}f(\vec n)=A_{r}f(\vec n)=\frac{1}{N(r)^{1-\frac{\beta}{d}}}{\sum_{\vec m\in \Omega_{r}}{f(\vec n+\vec m )}}.
\end{equation*}
Proceeding as in \eqref{Sec1_eq_conv_e_d} we find that 
\begin{align*}
\begin{split}
M_{\beta}f(\vec n + \vec{e}_d)&\geq\frac{1}{N(r+1)^{1-\frac{\beta}{d}}}{\sum_{\vec m\in \overline\Omega_{r+1}}{f(\vec n + \vec{e}_d +\vec m )}}\\
&\geq\frac{1}{N(r+1)^{1-\frac{\beta}{d}}}{\sum_{\vec m\in \overline\Omega_{r}}{f(\vec n +\vec m )}}.
\end{split}
\end{align*}
Hence
\begin{equation}\label{Arg_KiSa_eq1}
M_{\beta}f(\vec n) - M_{\beta}f(\vec n + \vec{e}_d) \leq \left( \frac{1}{N(r)^{1-\frac{\beta}{d}}} - \frac{1}{N(r+1)^{1-\frac{\beta}{d}}}\right) {\sum_{\vec m\in \overline\Omega_{r}}{f(\vec n +\vec m )}}.
\end{equation}
We claim that 
\begin{equation}\label{Arg_KiSa_eq2}
\left( \frac{1}{N(r)^{1-\frac{\beta}{d}}} - \frac{1}{N(r+1)^{1-\frac{\beta}{d}}}\right) \leq C\,  \frac{1}{N(r)^{1-\frac{\beta-1}{d}}}\,.
\end{equation}
This is certainly true for small $r$, whereas for large $r$ we may take $c_1$ in \eqref{Intro_Def_c1} and \eqref{Intro_Def_c1_2} to be strictly smaller than $\frac12$ and use the mean value theorem (after clearing the denominators). From \eqref{Arg_KiSa_eq1} and \eqref{Arg_KiSa_eq2} we find that 
\begin{equation*}
M_{\beta}f(\vec n) - M_{\beta}f(\vec n + \vec{e}_d) \leq C \,M_{\beta -1} f(\vec{n}).
\end{equation*}
If $M_{\beta} f(\vec{n}) < M_{\beta} f(\vec{n} + \vec{e}_d)$ an analogous argument leads to 
\begin{equation*}
M_{\beta}f(\vec n + \vec{e}_d) - M_{\beta}f(\vec n) \leq C \,M_{\beta -1} f(\vec{n} + \vec{e}_d).
\end{equation*}
Moreover, we may replace $\vec{e}_d$ by any other basis vector  $\vec{e}_j$. This leads to the following pointwise bound
\begin{equation}\label{Arg_KiSa_eq3}
\big|\nabla M_{\beta} f (\vec{n}) \big| \leq C \,\left\{M_{\beta -1} f(\vec{n}) + \sum_{j=1}^{d} M_{\beta -1} f(\vec{n} + \vec{e}_j)\right\}\,.
\end{equation}
Finally, letting $r = d/(d + \alpha -1)$ and $p^* = d/(d-1)$, we use the boundedness of $M_{\beta -1}$, interpolation and the Sobolev embedding (for which the proof in the discrete setting is analogous to the proof in the continuous setting as in \cite[Chapter V, \S 2.5]{S}) to obtain
\begin{align*}
\| \nabla M_{\beta} f \|_{\ell^q(\Z^d)} \leq C \| M_{\beta-1} f \|_{\ell^q(\Z^d)} \leq C' \| f \|_{\ell^r(\Z^d)} \leq C'  \| f \|_{\ell^{p^*}(\Z^d)}^{1 - \alpha}  \| f \|_{\ell^1(\Z^d)}^{\alpha} \leq C''  \|\nabla f \|_{\ell^{1}(\Z^d)}^{1 - \alpha}  \| f \|_{\ell^1(\Z^d)}^{\alpha}.
\end{align*}
This concludes the proof of this part.

\subsection{Uncentered case - part (ii)} The proof in this case is analogous to \S \ref{Cent_case_endpoint}, establishing the pointwise bound \eqref{Arg_KiSa_eq3} for the uncentered operator $\widetilde{M}_{\beta}$. We leave the details to the interested reader.

\section{Proof of Theorem \ref{Thm3} - Continuity} \label{Sec3}

In this section we keep working with the discrete maximal operators $M_{\beta} = M_{\Omega, \beta}$ (resp. $\wt{M}_{\beta} = \wt{M}_{\Omega, \beta}$). Given our convex set $\Omega$ as described in \S \ref{Subsec_Omega}, we fix the constants $C_\Omega = m(\Omega)$, $c_1$, $c_2$ and $\lambda$ as defined in \eqref{estimativa} - \eqref{Intro_Def_lambda}. 

\subsection{Centered case - part (i)} \label{Cont_cent_case_Part_i}
We want to show that if $f_{j}\to f$ in $\ell^{1}(\Z^{d})$ then $\nabla M_{\beta}f_{j}\to \nabla M_{\beta}f$ in $\ell^{q}(\Z^{d})$. From the fact that $\left||f_{j}|-|f|\right|\leq |f_{j}-f|$, we may assume without loss of generality that $f
_{j}\geq0$ for all $j$, and that $f\geq 0$. It suffices show that
\begin{equation}\label{Sec3_goal}
\left\|\frac{\partial}{\partial x_{i}}M_{\beta}f_{j}-\frac{\partial}{\partial x_{i}}M_{\beta}f\right\|_{\ell^{q}(\Z^{d})}\,\to\, 0
\end{equation}
 as $j\rightarrow \infty$, for each $i=1,2,\dots,d.$ As before, we will prove it for $i=d$ and the other cases are analogous.
 
\subsubsection{Convergence of radii} Given a function $g\in \ell^{1}(\Z^{d}) $ and a point  $\vec n\in \Z^{d},$ we first study the set of radii that realize the supremum in the fractional maximal function at the point $\vec n.$  Let us define 
\begin{equation*}
\RR g(\vec n)=\left\{r\in[0,\infty);\ M_{\beta}g(\vec n)=A_{r}|g|(\vec n)=\frac{1}{N(r)^{1-\frac{\beta}{d}}}\sum
_{\vec m\in\overline\Omega_{r}}|g(\vec n+\vec m)|\right\}.
\end{equation*}  
We prove the following auxiliary lemma, which can be seen as a discrete fractional analogue of a lemma of Luiro \cite[Lemma 2.2]{Lu1}.
\begin{lemma}\label{Luiro lemma}
Let $f_{j}\to f$ in $\ell^{1}(\Z^{d})$ and let $R>0$. There exists $j_{0}$ such that for $j\geq j_{0}$ we have $\RR f_{j}(\vec n)\subset \RR f(\vec n)$ for each $\vec n\in \overline B_{R}$. 
\end{lemma}

\begin{proof}
Fix $\vec n\in \overline B_{R}$ and consider the map $r \mapsto A_{r} f(\vec n)$. Since $f \in \ell^1(\Z^d)$, there is only a finite number of values in the image set $\{A_{r} f(\vec n); r \geq 0\}$ such that $A_{r} f(\vec n) \geq \tfrac{1}{2} M_{\beta}f (\vec{n})$. Therefore, there exists $\varepsilon(\vec{n})>0$ such that, if $A_{r} f(\vec n) > M_{\beta}f (\vec{n}) - \varepsilon(\vec{n})$, then $A_{r} f(\vec n) = M_{\beta}f (\vec{n})$ and $r \in \RR f(\vec n)$. Let 
\begin{equation*}
\varepsilon = \tfrac{1}{3}\min\big\{\varepsilon(\vec{n}); \  \vec n \in \overline B_{R}\big\}.
\end{equation*}
From \eqref{Intro_inclusion} we see that $\|f_j - f\|_{\ell^{1}(\Z^{d})} \to 0$ implies that $\|f_j - f\|_{\ell^{a}(\Z^{d})} \to 0$ for each $a \geq 1$. In particular, there exists $j_0$ such that $\|f_j-f\|_{\ell^{\frac{d}{\beta}}(\Z^{d})}<\varepsilon$ for all $j\geq j_{0}$ (if $\beta = 0$ then $d/\beta = \infty$). By H\"{o}lder's inequality,
\begin{align}\label{Sec3_eq0}
\begin{split}
\big|A_{r}f_{j}(\vec n)-A_{r}f(\vec n)\big|&\leq \frac{1}{N(r)^{1-\frac{\beta}{d}}}\sum_{\vec m\in\overline \Omega_{r}}\big|f_{j}(\vec n+\vec m)-f(\vec n+\vec m)\big|\\
&\leq \frac{1}{N(r)^{1-\frac{\beta}{d}}}\left(\sum_{\vec m\in \overline\Omega_{r}}\big|f_j{(\vec n+\vec m)}-f(\vec n+\vec m)\big|^{\frac{d}{\beta}} \right)^{\frac{\beta}{d}}N(r)^{1-\frac{\beta}{d}}\\
&\leq \|f_{j}-f\|_{\ell^{\frac{d}{\beta}}(\Z^{d})}\\
&\leq \varepsilon 
\end{split}
\end{align}
for all $r\geq 0$ and $j\geq j_{0}$. Hence, for any $\vec n\in \overline B_{R}$ and $j \geq j_0$, if we take $s \in \RR f(\vec n)$ we get
\begin{equation}\label{Sec3_eq1}
M_{\beta}f (\vec{n}) = A_{s} f(\vec n) = A_{s} f_j(\vec n) + A_{s} (f- f_j)(\vec n) \leq M_{\beta}f_j (\vec{n}) + \varepsilon.
\end{equation}
For any $\vec n\in \overline B_{R}$, $j \geq j_0$ and $r_j \in \RR f_j(\vec n)$, we use \eqref{Sec3_eq0} and \eqref{Sec3_eq1} to obtain
\begin{equation*}
A_{r_j} f(\vec n) = A_{r_j} (f - f_j)(\vec n) + A_{r_j} f_j(\vec n) \geq M_{\beta}f (\vec{n}) - 2\varepsilon.
\end{equation*}
From the definition of $\varepsilon$ we must have $r_j \in \RR f(\vec n)$, which completes the proof of the lemma.
\end{proof}

\subsubsection{Brezis-Lieb reduction} For a fixed $\vec n \in \Z^d$, given $\varepsilon >0$, we use Lemma \ref{Luiro lemma} to find $j_0$ such that, if $j \geq j_0$, we have $\RR f_{j}(\vec n)\subset \RR f(\vec n)$ and $\|f_j-f\|_{\ell^{\frac{d}{\beta}}(\Z^{d})}<\varepsilon$. Taking $r_j \in \RR f_{j}(\vec n)$ and using \eqref{Sec3_eq0} we get
\begin{equation*}
|M_{\beta}f_j(\vec n)-M_{\beta}f(\vec n)|=|A_{r_{j}}f_j(\vec n)-A_{r_{j}}f(\vec n)|\leq \varepsilon
\end{equation*}
for $j \geq j_0$, and therefore  $M_{\beta}f_{j}(\vec n)\to M_{\beta}f(\vec n)$ as $j\to \infty$. Hence it follows that
\begin{equation}\label{point conv}
\frac{\partial}{\partial x_{d}}M_{\beta}f_{j}(\vec n)\rightarrow\frac{\partial}{\partial x_{d}} M_{\beta}f(\vec n)
\end{equation} 
as $j\to\infty$. From the classical Brezis-Lieb lemma \cite{BL}, in order to prove \eqref{Sec3_goal} it suffices to prove the convergence of the norms, i.e.
\begin{equation*}
\lim_{j\to\infty}\left\|\frac{\partial}{\partial x_{d}}M_{\beta}f_{j}\right\|_{\ell^{q}(\Z^{d})}=\,\left\|\frac{\partial}{\partial x_{d}}M_{\beta}f\right\|_{\ell^{q}(\Z^{d})}.
\end{equation*}

\smallskip

From Fatou's lemma and the pointwise convergence in \eqref{point conv} we have
\begin{equation*}
\left\|\frac{\partial}{\partial x_{d}}M_{\beta}f\right\|_{\ell^{q}(\Z^{d})}\,\leq\, \liminf_{j\to\infty}\left\|\frac{\partial}{\partial x_{d}}M_{\beta}f_{j}\right\|_{\ell^{q}(\Z^{d})}.
\end{equation*} 
We prove the opposite inequality in the next subsection.

\subsubsection{Upper bound} Let $\varepsilon_0 >0$ be given. Our goal now is to find $j_0$ such that, for $j \geq j_0$, 
\begin{equation}\label{Sec3_Goal}
\left\|\frac{\partial}{\partial x_{d}}M_{\beta}f_{j}\right\|_{\ell^{q}(\Z^{d})} \leq \left\|\frac{\partial}{\partial x_{d}}M_{\beta}f\right\|_{\ell^{q}(\Z^{d})} + \varepsilon_0.
\end{equation} 
Let $R$ be a sufficiently large radius (to be properly chosen later) and let $Q_{2R}=\{\vec x\in\R^{d}:\|\vec x\|_{\infty}\leq 2R\}$ be the cube of side $4R$. We write
\begin{align}\label{Sec3_initial_split}
\begin{split}
\left\|\frac{\partial}{\partial x_{d}}M_{\beta}f_j\right\|_{\ell^{q}(\Z^{d})}^q & = \sum_{\|\vec n\|_{\infty} \leq 2R}  \left|\frac{\partial}{\partial x_{d}}M_{\beta}f_j(\vec n)\right|^q + \sum_{\|\vec n\|_{\infty} > 2R}  \left|\frac{\partial}{\partial x_{d}}M_{\beta}f_j(\vec n)\right|^q\\& =: S_1 + S_2.
\end{split}
\end{align}
We bound $S_1$ and $S_2$ separately. 

\smallskip

Let $\varepsilon_{1}>0$ (to be properly chosen later). By Lemma \ref{Luiro lemma}, there exists $j_{1}$ such that, if $j \geq j_1$, we have $\RR f_{j}(\vec n)\subset \RR f(\vec n)$ for each $\vec n$ with $\|\vec n\|_{\infty}\leq 2R+1$ and
 \begin{equation*}
 \|f_{j}-f\|_{\ell^{\frac{d}{\beta}}(\Z^{d})}\leq \varepsilon_{1}.
 \end{equation*}
By \eqref{Sec3_eq0} we obtain that
\begin{equation*}
\left|\frac{\partial}{\partial x_{d}}M_{\beta}f_{j}(\vec n)-\frac{\partial}{\partial x_{d}}M_{\beta}f(\vec n)\right|\leq 2\varepsilon_{1},
\end{equation*}
for any $\vec n\in Q_{2R}$. Hence, by the triangle inequality,
 \begin{equation}\label{Sec3_bound_S_1}
S_1 \leq \left(\left\| \frac{\partial}{\partial x_{d}}M_{\beta}f\right\|_{\ell^{q}(Q_{2R})}+2\varepsilon_{1}(4R+1)^{\frac{d}{q}}\right)^q.
 \end{equation}
 
 \smallskip
 
 In order to bound $S_2$, we recall the definition of the sets $X_k^{+}$ and $X_k^{-}$ in \eqref{Def_X_n_+} and \eqref{Def_X_n_-}. We then write
 \begin{align*}
 S_2 & =  \sum_{k\geq0}\sum_{\substack{\|\vec n \|_{\infty} > 2R \\ \vec n \in X^{+}_{k}}}{\big(M_{\beta} f_j(\vec n + \vec{e}_d)- M_{\beta} f_j(\vec n)}\big)^{q} + \sum_{k\geq0}\sum_{\substack{\|\vec n \|_{\infty} > 2R \\ \vec n \in X^{-}_{k}}}{\big(M_{\beta} f_j(\vec n)- M_{\beta} f_j(\vec n + \vec{e}_d)}\big)^{q}\\
 & =: S_2^{+} + S_2^{-}.
 \end{align*}
From the calculations leading to \eqref{Sec2_Important_estimate} we have
\begin{align*}
\begin{split}
S_2^{-} &\  \leq  \ \|\nabla f_j\|_{\ell^1(\Z^d)}^{(q-1)(1-\alpha)} \left[ \sum_{\vec m\in \Z^d}  \frac{1}{N^{+}\big(\frac{|\vec{m}|}{\lambda}-1\big)^{1+\gamma}}\sum_{k \geq \frac{|\vec m|}{\lambda}}\sum_{\substack{\|\vec n \|_{\infty} > 2R \\ \vec n \in X^{-}_{k}}} \big|f_j(\vec n+\vec m )-f_j(\vec n + \vec{e}_d +\vec m)\big| \right]^{1-\alpha}\\
& \ \ \ \ \ \ \ \ \ \ \ \ \ \ \ \ \  \ \ \ \ \ \ \times \| f_j\|_{\ell^1(\Z^d)}^{(q-1)\alpha} \left[ \sum_{\vec m\in \Z^d} \frac{C}{N^{+}\big( \frac{|\vec m|}{\lambda} -1\big)^{1+\gamma}} \sum_{k \geq \frac{|\vec m|}{\lambda}} \sum_{\substack{\|\vec n \|_{\infty} > 2R \\ \vec n \in X^{-}_{k}}}f_j(\vec n+\vec m)\right]^{\alpha}.
\end{split}
\end{align*}
We have already noted after \eqref{ineq holder_with_q-1_v4} that the sum $\sum_{\vec m\in \Z^d}  \frac{1}{N^{+}\big(\frac{|\vec{m}|}{\lambda}-1\big)^{1+\gamma}}$ is finite. We now consider two cases: (i) when $\|\vec m \|_{\infty} < R$, which implies that $\vec n + \vec m \in Q_R^{c}$; (ii) $\|\vec m \|_{\infty} \geq R$. We define the function
\begin{equation*}
h(R) = \sum_{\|\vec m\|_{\infty} \geq R} \frac{1}{N^{+}\big(\frac{|\vec{m}|}{\lambda}-1\big)^{1+\gamma}}\,,
\end{equation*}
and obtain the following upper bound (recall that the sets $X_k^{-}$ are pairwise disjoint)
\begin{align}\label{Sec3_bound_for_S_2_-}
\begin{split}
S_2^{-} &\  \leq  \ \|\nabla f_j\|_{\ell^1(\Z^d)}^{(q-1)(1-\alpha)}\left( C \,\|\nabla f_j\|_{\ell^1(Q_R^{c})} + h(R) \|\nabla f_j\|_{\ell^1(\Z^d)}\right)^{(1-\alpha)} \\
&  \ \ \ \ \ \ \ \ \ \ \ \ \ \ \ \ \ \  \times  \|f_j\|_{\ell^1(\Z^d)}^{(q-1)(\alpha)}\left( C \,\| f_j\|_{\ell^1(Q_R^{c})} + h(R) \| f_j\|_{\ell^1(\Z^d)}\right)^{\alpha}.
\end{split}
\end{align}
The same bound as in \eqref{Sec3_bound_for_S_2_-} holds for $S_2^{+}$. Putting together \eqref{Sec3_initial_split}, \eqref{Sec3_bound_S_1} and \eqref{Sec3_bound_for_S_2_-} (this last one duplicated, for $S_2^{-}$ and for $S_2^{+}$) we obtain 
\begin{align*}
\left\|\frac{\partial}{\partial x_{d}}M_{\beta}f_{j}\right\|_{\ell^{q}(\Z^{d})}^q & \leq \left(\left\| \frac{\partial}{\partial x_{d}}M_{\beta}f\right\|_{\ell^{q}(Q_{2R})}+2\varepsilon_{1}(4R+1)^{\frac{d}{q}}\right)^q \\
&  \ \ \ \ \ \ \ \ \ +  2  \left\{\|\nabla f_j\|_{\ell^1(\Z^d)}^{(q-1)(1-\alpha)}\left( C \,\|\nabla f_j\|_{\ell^1(Q_R^{c})} + h(R) \|\nabla f_j\|_{\ell^1(\Z^d)}\right)^{(1-\alpha)} \right.\\
&  \ \ \ \ \ \ \ \ \ \ \ \ \ \ \ \ \ \  \left. \times   \|f_j\|_{\ell^1(\Z^d)}^{(q-1)\alpha}\left( C \,\| f_j\|_{\ell^1(Q_R^{c})} + h(R) \| f_j\|_{\ell^1(\Z^d)}\right)^{\alpha}\right\}.
\end{align*}
The crucial point in our argument is the fact that $h(R) \to 0$ as $R \to \infty$. It is now clear that we can choose $R$ sufficiently large to make the second term above very small (for $j$ large), and then we choose $\varepsilon_1$ sufficiently small to arrive at \eqref{Sec3_Goal}. This completes the proof in the centered case.

\subsection{Uncentered case - part (i)} \label{Cont_uncent_case_partI_sectionline} Minor modifications are needed in comparison to the previous argument. For a function $g\in \ell^{1}(\Z^{d})$ and a point $\vec n\in\Z^{d}$, we now define 
\begin{equation*}
\widetilde{\RR}g(\vec n)=\left\{(\vec x,r)\in \R^{d}\times\R^{+}; \, \widetilde M_{\beta}\,g(\vec n)=A_{(\vec x,r)}|g|(\vec n)=\frac{1}{N(\vec x,r)^{1-\frac{\beta}{d}}}\sum_{\vec m\in\overline\Omega_{r}(\vec x)} |g(\vec m)|\right\}.
\end{equation*}
The following lemma can be proved with the same ideas used in the proof of Lemma \ref{Luiro lemma}.
\begin{lemma}
Let $f_{j}\to f$ in $\ell^{1}(\Z^{d})$ and let $R>0$. There exists $j_{0}$ such that for $j\geq j_{0}$ we have $\wt{\RR} f_{j}(\vec n)\subset \wt{\RR} f(\vec n)$ for each $\vec n\in \overline B_{R}$. 
\end{lemma} 
\noindent Once we have adjusted this auxiliary lemma, the remaining steps of the proof are analogous to the centered case.

\subsection{Centered case - part (ii)}\label{Cont_cent_case_partII_sectionline} The proof follows along similar lines to \S \ref{Cont_cent_case_Part_i} and we only present the minor modifications needed. Observe that Lemma \ref{Luiro lemma} continues to hold \footnote{In fact, note that all we need to establish Lemma \ref{Luiro lemma} is that $f \in \ell^s(\Z^d)$ and $f_j \to f$ in $\ell^s(\Z^d)$ for some $1\leq s < d/\beta$.} and we may arrive at \eqref{Sec3_initial_split} in the same way. We bound $S_1$ just as we did in \eqref{Sec3_bound_S_1}. To bound $S_2$ we need a different argument, since $\gamma$ would be zero in this case. In fact, we use the pointwise estimate \eqref{Arg_KiSa_eq3} to get directly
\begin{equation}\label{Cont_cent_case_part_ii}
S_2 = \sum_{\|\vec n\|_{\infty} > 2R}  \left|\frac{\partial}{\partial x_{d}}M_{\beta}f_j(\vec n)\right|^q \leq C \sum_{\|\vec n\|_{\infty} \geq 2R}  \big|M_{\beta-1}f_j(\vec n)\big|^q.
\end{equation}
Since $f_j \to f$ in $\ell^1(\Z^d)$ we have $f_j \to f$ in $\ell^r(\Z^d)$ for $r = d/(d + \alpha -1)$. Since $M_{\beta -1}: \ell^r(\Z^d) \to \ell^q(\Z^d)$ is bounded and continuous we have $M_{\beta -1} f_j \to M_{\beta -1} f$ in $\ell^q(\Z^d)$. We can then choose $R$ sufficiently large to make the right-hand side of \eqref{Cont_cent_case_part_ii} very small (for $j$ large). The rest of the proof is analogous to \S \ref{Cont_cent_case_Part_i}.

\subsection{Uncentered case - part (ii)} The proof is essentially analogous to \S \ref{Cont_cent_case_Part_i}, \S \ref{Cont_uncent_case_partI_sectionline} and \S \ref{Cont_cent_case_partII_sectionline}. One just has to use the analogue of \eqref{Arg_KiSa_eq3} to the uncentered operator to bound $S_2$ as in \eqref{Cont_cent_case_part_ii}. We omit the details.

\section{Proof of Theorem \ref{Thm2}} \label{Sec4}

In this section we work with the one-dimensional discrete uncentered maximal operator defined in \eqref{Intro_Def_Disc_Max_oper_dim1}. To prove Theorem \ref{Thm2} we may assume without loss of generality that $f\geq0$. For $n\in\Z$ and $r,s \in \Z^+ \times \Z^+$ we define the fractional average 
$$
A_{r,s}f(n)=\frac{1}{(r+s+1)^{1-\beta}}\sum_{k=-r}^{s}f(n+k).
$$
We start with the following preliminary lemma.

\begin{lemma}\label{Sec4_Lem6}
Let $f:\Z \to \R^+$ be a bounded function such that $\wt{M}_{\beta}f \not\equiv \infty$. 
\begin{enumerate}
\item[(i)] We have $\wt{M}_{\beta}f(n) <\infty$ for all $n \in \Z$.

\smallskip

\item[(ii)] If $\wt{M}_{\beta}f(n)$ is not attained by any average $A_{r,s}f(n)$ with $r,s \in \Z^+ \times \Z^+$, then 
\begin{equation}\label{Sec4_lem6_eq1}
\wt{M}_{\beta}f(m) \geq \wt{M}_{\beta}f(n)
\end{equation}
for all $m \in \Z$. 
\end{enumerate}
\end{lemma}
\begin{proof}
(i) If there is $n \in \Z$ such that $\wt{M}_{\beta}f(n) =\infty$, there exists a sequence $\{r_j, s_j\}$ in $\Z^+ \times \Z^+$, with $r_j + s_j \to \infty$ such that $A_{r_j,s_j}f(n) \to \infty$ as $j \to \infty$. For each $m \in \Z$, letting $C = \|f\|_{\ell^{\infty}(\Z)}$ we have
\begin{equation}\label{Sec4_lem6_eq2}
A_{r_j,s_j}f(m) \geq A_{r_j,s_j}f(n) - \frac{2 C |m-n|}{(r_j+s_j+1)^{1-\beta}}\,,
\end{equation}
which implies that $\wt{M}_{\beta}f(m) =\infty$, a contradiction.

\smallskip

\noindent(ii) If $\wt{M}_{\beta}f(n)$ is not attained, there exists a sequence $\{r_j, s_j\}$ in $\Z^+ \times \Z^+$, with $r_j + s_j \to \infty$ such that $A_{r_j,s_j}f(n) \to \wt{M}_{\beta}(n)$ as $j \to \infty$. The inequality \eqref{Sec4_lem6_eq1} plainly follows from \eqref{Sec4_lem6_eq2}.
\end{proof}

The next lemma is the heart of the proof. It bounds the $q$-variation of $\wt{M}_{\beta}f$ in a monotone interval by the variation of $f$ in a comparable interval.
\begin{lemma}\label{Lem7_crucial}
Let $f:\Z \to \R^+$ be a function of bounded variation such that $\wt{M}_{\beta}f$ is non-constant $($in particular, $\wt{M}_{\beta}f \not\equiv \infty$$)$.
\begin{enumerate}
\item[(i)] Let $a < b$ be integers such that $\wt{M}_{\beta}f$ is non-increasing in $[a,b]$, with $\wt{M}_{\beta}f(a) > \wt{M}_{\beta}f(a+1)$. Let $r$ be the smallest nonnegative integer such that $\wt{M}_{\beta}f(a) = A_{r,0}f(a)$. Then we have
\begin{align}\label{Sec5_Lem_novo_eq1}
\sum_{n=a}^{b-1}\big|\widetilde M_{\beta}f(n)-\widetilde M_{\beta}f(n+1)\big|^{q} \leq 2\,\| f'\|_{\ell^{1}(\Z)}^{q-1}\sum_{n=a-r}^{b-1}|f(n)-f(n+1)|.
\end{align}

\item[(ii)] Let $a < b$ be integers such that $\wt{M}_{\beta}f$ is non-decreasing in $[a,b]$, with $\wt{M}_{\beta}f(b-1) < \wt{M}_{\beta}f(b)$. Let $s$ be the smallest nonnegative integer such that $\wt{M}_{\beta}f(b) = A_{0,s}f(b)$. Then we have
\begin{align}\label{Sec5_Lem_novo_eq2}
\sum_{n=a}^{b-1}\big|\widetilde M_{\beta}f(n)-\widetilde M_{\beta}f(n+1)\big|^{q} \leq 2\,\| f'\|_{\ell^{1}(\Z)}^{q-1}\sum_{n=a}^{b+s-1}|f(n)-f(n+1)|.
\end{align}
\end{enumerate}
\end{lemma}
\begin{proof}
We prove (i) and (ii) is analogous. Observe first that the existence of such minimal $r$ is guaranteed by Lemma \ref{Sec4_Lem6}. Then note that 
\begin{align}\label{Sec4_eq1_M_M}
\begin{split}
\big|\widetilde M_{\beta}f(a)-\widetilde M_{\beta}f(b)\big|&\leq \big| A_{r,0}f(a)- A_{r,0}f(b)\big| =\left| \sum_{k=0}^{b-a-1}A_{r,0}f(a+k)-A_{r,0}f(a+k+1)\right|\\
&\leq\frac{1}{(r+1)^{1-\beta}}\sum_{n=a-r}^{b-1}(r+1)|f(n)-f(n+1)|\\
&=(r + 1)^{\beta} \sum_{n=a-r}^{b-1}|f(n)-f(n+1)|.
\end{split}
\end{align}
Let  $m$ be the smallest integer in $[a,b-1]$ such that 
$$\widetilde M_{\beta}f(m)-\widetilde M_{\beta}f(m+1) =  \max\big\{\widetilde M_{\beta}f(n)-\widetilde M_{\beta}f(n+1); \,n \in [a, b-1]\big\} > 0\,,$$
and let $t=\min\{t\in \Z^+;\,\widetilde M_{\beta}f(m)=A_{t,0}f(m) \}$. The existence of such $t$ is guaranteed by Lemma \ref{Sec4_Lem6}.

\subsubsection*{Step 1} Let us first consider the situation when $t \geq r$. In this case, using \eqref{Sec4_eq1_M_M} we obtain
\begin{align}\label{Sec4_eq1_M_M_m}
\begin{split}
\sum_{n=a}^{b-1}\big|\widetilde M_{\beta}f(n)-\widetilde M_{\beta}f(n+1)\big|^{q}&\leq \big|\widetilde M_{\beta}f(m)-\widetilde M_{\beta}f(m+1)\big|^{q-1}\sum_{n=a}^{b-1}\big|\widetilde M_{\beta}f(n)-\widetilde M_{\beta}f(n+1)\big|\\
&=\big|\widetilde M_{\beta}f(m)-\widetilde M_{\beta}f(m+1)\big|^{q-1}\,\big|\widetilde M_{\beta}f(a)-\widetilde M_{\beta}f(b)\big|\\
&\leq\big|\widetilde M_{\beta}f(m)-\widetilde M_{\beta}f(m+1)\big|^{q-1}(r + 1)^{\beta} \sum_{n=a-r}^{b-1}|f(n)-f(n+1)|\\
&=\big|(r+1)^{1-\beta}\big(\widetilde M_{\beta}f(m)-\widetilde M_{\beta}f(m+1)\big)\big|^{q-1}\sum_{n=a-r}^{b-1}|f(n)-f(n+1)|\\
&\leq\big|(r +1)^{1-\beta}\big(A_{t,0}f(m)-A_{t,0}f(m+1)\big)\big|^{q-1}\sum_{n=a-r}^{b-1}|f(n)-f(n+1)|\\
&\leq\left|\frac{(r +1)^{1-\beta}}{(t+1)^{1-\beta}}\big(f(m-t)-f(m+1)\big)\right|^{q-1}\sum_{n=a-r}^{b-1}|f(n)-f(n+1)|\\
&\leq\| f'\|_{\ell^{1}(\Z)}^{q-1}\sum_{n=a-r}^{b-1}|f(n)-f(n+1)|.
\end{split}
\end{align}
This establishes \eqref{Sec5_Lem_novo_eq1} in this case.

\subsubsection*{Step 2} Now assume that $t< r$ and $m-t \leq a$. Note that in this case we have $a - r < m - t$. We may proceed as in \eqref{Sec4_eq1_M_M} and \eqref{Sec4_eq1_M_M_m} to obtain
\begin{align}\label{Sec4_eq1_M_M_step2}
\begin{split}
&\sum_{n=a}^{b-1}  \big|\widetilde M_{\beta}f(n)-\widetilde M_{\beta}f(n+1)\big|^{q}\\
&\leq (t+1) \big|\widetilde M_{\beta}f(m)-\widetilde M_{\beta}f(m+1)\big|^{q} + \big|\widetilde M_{\beta}f(m)-\widetilde M_{\beta}f(m+1)\big|^{q-1}\sum_{n=m}^{b-1}\big|\widetilde M_{\beta}f(n)-\widetilde M_{\beta}f(n+1)\big|\\
&\leq|f(m-t)-f(m+1)|^{q} +\big|\widetilde M_{\beta}f(m)-\widetilde M_{\beta}f(m+1)\big|^{q-1}\big|\widetilde M_{\beta}f(m)-\widetilde M_{\beta}f(b)\big|\\
&\leq|f(m-t)-f(m+1)|^{q}+\big|\widetilde M_{\beta}f(m)-\widetilde M_{\beta}f(m+1)\big|^{q-1}\,(t +1)^{\beta}\sum_{n=m-t}^{b-1}|f(n)-f(n+1)|\\
&\leq\left|f(m-t)-f(m+1)\right|^{q-1} \left(|f(m-t)-f(m+1)|+\sum_{n=m-t}^{b-1}|f(n)-f(n+1)|\right)\\
& \leq 2\| f'\|_{\ell^{1}(\Z)}^{q-1}\sum_{n=a-r}^{b-1}|f(n)-f(n+1)|,
\end{split}
\end{align}
and we have again established \eqref{Sec5_Lem_novo_eq1}.

\subsubsection*{Step 3} Finally, we consider the case $t< r$ and $a < m-t$. Reasoning as in \eqref{Sec4_eq1_M_M_step2} we obtain
\begin{align}\label{Sec4_induction_eq0}
\begin{split}
&\sum_{n=m-t}^{b-1}  \big|\widetilde M_{\beta}f(n)-\widetilde M_{\beta}f(n+1)\big|^{q}\\
&\leq (t+1) \big|\widetilde M_{\beta}f(m)-\widetilde M_{\beta}f(m+1)\big|^{q} + \big|\widetilde M_{\beta}f(m)-\widetilde M_{\beta}f(m+1)\big|^{q-1}\sum_{n=m}^{b-1}\big|\widetilde M_{\beta}f(n)-\widetilde M_{\beta}f(n+1)\big|\\
&\leq 2\| f'\|_{\ell^{1}(\Z)}^{q-1}\sum_{n=m-t}^{b-1}|f(n)-f(n+1)|.
\end{split}
\end{align}
We then proceed inductively. Let $(m_1,t_1) = (m,t)$. Having defined $(m_{1}, t_1), (m_{2}, t_2), \ldots, (m_{l-1},t_{l-1})$, if $t_{l-1} < r$ and $a< m_{{l-1}} - t_{{l-1}}$ we define $m_{l}$ as the smallest integer in the interval $[a,m_{{l-1}}-t_{{l-1}}-1]$ such that 
$$\widetilde M_{\beta}f(m_l)-\widetilde M_{\beta}f(m_l+1) =  \max\big\{\widetilde M_{\beta}f(n)-\widetilde M_{\beta}f(n+1); \,n \in[a,m_{{l-1}}-t_{{l-1}}-1]\big\} > 0\,,$$
and let $t_{l}=\min\{t\in \Z^+;\,\widetilde M_{\beta}f(m_{l})=A_{t,0}f(m_{l}) \}$. If $t_{l} < r$ and $a< m_{{l}} - t_{{l}}$ we reboot Step 3 to obtain
\begin{align}\label{Sec4_induction}
\sum_{n=m_{{l}}-t_{{l}}}^{m_{{l-1}}-t_{{l-1}}-1}\big|\widetilde M_{\beta}f(n)-\widetilde M_{\beta}f(n+1)\big|^{q}\leq2\,\| f'\|_{\ell^{1}(\Z)}^{q-1}\sum_{n=m_{{l}}-t_{{l}}}^{m_{{l-1}}-t_{{l-1}}-1}|f(n)-f(n+1)|.
\end{align} 
This process must terminate, i.e. there exists a smallest $N$ such that either (i) $t_N \geq r$ or (ii) $t_N < r$ and $m_{N} - t_{N} \leq a$. In the first case we use Step 1 to bound the $q$-variation of $\widetilde M_{\beta}f$ on the interval $[a, m_{N-1} - t_{N-1}]$ and in the second case we use Step 2 to bound the $q$-variation of $\widetilde M_{\beta}f$ on the interval $[a, m_{N-1} - t_{N-1}]$. We then sum with all the previous inequalities \eqref{Sec4_induction_eq0} and \eqref{Sec4_induction} of the inductive process to arrive at the desired result.
\end{proof}

We now introduce the local maxima and minima of a discrete function $g:\Z \to \R$.\footnote{The local extrema are defined slightly differently in  \cite{BCHP, CH}, but used with the meaning stated here.} We say that an interval $[n,m]$ is a {\it string of local maxima} of $g$ if 
$$g(n-1) < g(n) = \ldots = g(m) > g(m+1).$$
If $n = -\infty$ or $m = \infty$ (but not both simultaneously) we modify the definition accordingly, eliminating one of the inequalities. The rightmost point $m$ of such a string is a {\it right local maximum} of $g$, while the leftmost point $n$ is a {\it left local maximum} of $g$. We define {\it string of local minima}, {\it right local minimum} and {\it left local minimum} analogously.

\begin{proof}[Proof of Theorem \ref{Thm2}] 
Assume that $\wt{M}_{\beta}f$ is not constant (in case $\wt{M}_{\beta}f$ is constant the result is obviously true). Let $\{[a_j^-,a_j^+]\}_{j \in \Z}$ and $\{[b_j^-,b_j^+]\}_{j \in \Z}$ be the ordered strings of local maxima and local minima of $\wt{M}_{\beta}f$ (we allow the possibilities of $a_j^{-}$ or $b_j^{-} = - \infty$ and $a_j^{+}$ or $b_j^{+} =  \infty$), i.e. 
\begin{equation}\label{Sec4_sequence}
\ldots < a_{-1}^- \leq a_{-1}^+ < b_{-1}^- \leq b_{-1}^+ < a_0^- \leq a_0^+ < b_0^-\leq b_0^+ < a_1^- \leq a_1^+ < b_1^- \leq b_1^+ < \ldots
\end{equation}
This sequence may terminate in one or both sides (in principle, we are not even ruling out the possibility of this sequence being empty, i.e. of $\wt{M}_{\beta}f$ being monotone), and we adjust the notation accordingly.

\smallskip

Let $a_j^+$ be one of the right local maxima and $r_j^+$ be the smallest integer radius such that $\wt{M}_{\beta}f(a_j^+) = A_{r_j^+,0}f(a_j^+)$. The crucial observation is that $a_j^- \leq a_j^+ - r_j^+$ and Lemma \ref{Lem7_crucial} yields
\begin{align*}
\sum_j \sum_{n=a_j^+}^{b_j^- -1}\big|\widetilde M_{\beta}f(n)-\widetilde M_{\beta}f(n+1)\big|^{q} \leq 2\,\| f'\|_{\ell^{1}(\Z)}^{q-1}\sum_j \sum_{n=a_j^-}^{b_j^--1}|f(n)-f(n+1)|.
\end{align*}
Analogously,
\begin{align*}
\sum_j \sum_{n=b_j^+}^{a_{j+1}^- -1}\big|\widetilde M_{\beta}f(n)-\widetilde M_{\beta}f(n+1)\big|^{q} \leq 2\,\| f'\|_{\ell^{1}(\Z)}^{q-1}\sum_j \sum_{n=b_j^+}^{a_{j+1}^+-1}|f(n)-f(n+1)|.
\end{align*}
Adding up these inequalities we obtain the desired result (note the potential overlap in each $[a_j^-,a_j^+]$)
$$\sum_{n=-\infty}^{\infty} \big|\widetilde M_{\beta}f(n)-\widetilde M_{\beta}f(n+1)\big|^{q} \leq 4\, \| f'\|_{\ell^{1}(\Z)}^{q}.$$
Note that, in the exceptional cases when the sequence \eqref{Sec4_sequence} terminates to one or both sides, or even when the sequence \eqref{Sec4_sequence} is empty (if $\wt{M}_{\beta}f$ is monotone), the $q$-variation of $\wt M_{\beta}f$ on the intervals of length infinity where it is monotone can be bounded directly using Lemma \ref{Lem7_crucial}. 
\end{proof}

\section{Proof of Theorem \ref{Thm1}} \label{Sec5}

We now move the discussion to the continuous setting. In this case, the one-dimensional uncentered  fractional maximal operator $\wt{M}_{\beta}$ is defined as in \eqref{Intro_frac_max_op}. The ideas presented in the previous section (discrete setting) also play a relevant role here, while new technical details arise. When proving Theorem \ref{Thm1} we may assume without loss of generality that $f \geq 0$ since $\var(f) \leq \var(|f|)$. The case $\beta =0$ of Theorem \ref{Thm1} was proved by Aldaz and P\'{e}rez L\'{a}zaro \cite{AP} (with the sharp constant $C=1$), so throughout this section we restrict ourselves to the case $0 < \beta < 1$. For $r,s \geq 0$ we keep denoting the fractional averages by \footnote{We define $A_{0,0} (f)(x) =  \limsup_{r,s \to 0^+}\frac{1}{(r+s)^{1 - \beta}} \int_{x-r}^{x+s} f(y)\,\dy$. If $f$ is locally bounded and $\beta >0$ we have $A_{0,0} (f)(x) = 0$.}
\begin{equation*}
 A_{r,s}f(x) = \frac{1}{(r+s)^{1 - \beta}} \int_{x-r}^{x+s} f(y)\,\dy.
\end{equation*}
We start with the following preliminary lemma.
\begin{lemma}\label{Sec4_Lem8_prelim}
Let $0 < \beta < 1$. Let $f:\R \to \R^+$ be a bounded function such that $\wt{M}_{\beta}f \not\equiv \infty$. 
\begin{enumerate}
\item[(i)] We have $\wt{M}_{\beta}f(x) <\infty$ for all $x \in \R$.

\smallskip

\item[(ii)] If $\wt{M}_{\beta}f(x)$ is not attained by any average $A_{r,s}f(x)$ with $r,s \geq 0$, then 
\begin{equation}\label{Sec5_lem6_eq1}
\wt{M}_{\beta}f(y) \geq \wt{M}_{\beta}f(x)
\end{equation}
for all $y \in \R$. 
\end{enumerate}
\end{lemma}

\begin{proof}
(i) If $\wt{M}_{\beta}f(x) = \infty$ for some $x \in \R$, since $f$ is bounded there exists a sequence $(r_j,s_j) \in \R^+ \times \R^+$ with $(r_j+s_j) \to \infty$ such that $A_{r_j,s_j}f(x) \to \infty$ as $j \to \infty$. Letting $C = \|f\|_{L^{\infty}(\R)}$, for any $y \in \R$ we have
\begin{equation}\label{Sec5_eq2_a_rs}
A_{r_j,s_j}f(y) \geq A_{r_j,s_j}f(x) - \frac{2 C |x-y|}{(r_j + s_j)^{1 - \beta}},
\end{equation}
which implies that $\wt{M}_{\beta}f(y) = \infty$, a contradiction.

\smallskip

\noindent (ii) If $\wt{M}_{\beta}f(x)$ is not attained, there exists a sequence $(r_j,s_j) \in \R^+ \times \R^+$ such that $A_{r_j,s_j}f(x) \to \wt{M}_{\beta}f(x)$ as $j \to \infty$. If $(r_j+s_j)$ is bounded, passing to a subsequence if necessary, we can find $(r,s) \in \R^+ \times \R^+$ such that $r_j \to r$ and $s_j \to s$, which implies that $A_{r,s}f(x) = \wt{M}_{\beta}f(x)$, a contradiction. Hence we have $(r_j+s_j) \to \infty$ and inequality \eqref{Sec5_lem6_eq1} plainly follows from \eqref{Sec5_eq2_a_rs}. 
\end{proof}

Our next proposition establishes the result for Lipschitz functions.

\begin{proposition}\label{q variation L}
Let $0 < \beta < 1$. Let $f:\R\rightarrow\R$ be a Lipschitz function such that $\wt{M}_{\beta}f \not\equiv \infty$. Then 
\begin{equation}\label{Sec5_Lem8_eq1}
\var_q\big(\wt{M}_{\beta}f\big) \leq 8^{1/q}\,\var(f).
\end{equation}
\end{proposition} 
We postpone the proof of Proposition \ref{q variation L} until \S \ref{Proof_Lipschitz_case}. For now, let us assume the validity of this result to conclude the proof of Theorem \ref{Thm1}. In order to do so, we also need the following classical result of F. Riesz (see \cite[Chapter IX \S4, Theorem 7]{N}).
\begin{proposition}[F. Riesz]\label{Prop_Riesz}
Let $g:\R \to \R$ be a given function and $1 < q < \infty$. Then $\var_q(g) < \infty$ if and only if $g$ is absolutely continuous and its derivative $g'$ belongs to $L^q(\R)$. Moreover, in this case, we have that
\begin{equation}\label{Riesz_new_id}
\|g'\|_{L^q(\R)} = \var_q(g).
\end{equation}
\end{proposition} 

\subsection{Proof of Theorem \ref{Thm1}} We start by showing the validity of the inequality in \eqref{Intro_q_bound1} for a function $f:\R \to \R^+$ of bounded variation such that $\wt{M}_{\beta}f \not\equiv \infty$. Let $\varphi \in C^{\infty}_{c}(\R)$ be a nonnegative smooth function with support in $[-1,1]$ and $\|\varphi\|_{L^1(\R)} =1$. For $\varepsilon >0$ we define $\varphi_{\varepsilon}(x) = \frac{1}{\varepsilon} \varphi(\frac{x}{\varepsilon})$ and write $f_{\varepsilon} = f * \varphi_{\varepsilon}$. Note that $f_{\varepsilon}$ is Lipschitz continuous and $\var(f_{\varepsilon}) \leq \var(f)$ for all $\varepsilon >0$.
\smallskip

Fix a partition $\mc{P} = \{x_1 < x_2 < \ldots < x_N\}$. By Proposition \ref{q variation L} we have
\begin{equation}\label{Sec5_var_q_eps}
\left(\sum_{n=1}^{N-1} \frac{\big|\wt{M}_{\beta} f_{\varepsilon}(x_{n+1}) - \wt{M}_{\beta}f_{\varepsilon}(x_n)\big|^q}{|x_{n+1} - x_n|^{q-1} } \right)^{1/q}\leq 8^{1/q} \,\var(f_{\varepsilon}) \leq 8^{1/q} \,\var(f)
\end{equation}
for all $\varepsilon >0$. We claim that 
\begin{equation}\label{Sec5_point_conv}
\lim_{\varepsilon \to 0} \wt{M}_{\beta} f_{\varepsilon}(x) = \wt{M}_{\beta} f(x)
\end{equation}
for all $x \in \R$. It then follows from \eqref{Sec5_var_q_eps} and \eqref{Sec5_point_conv} that 
\begin{equation*}
\left(\sum_{n=1}^{N-1} \frac{\big|\wt{M}_{\beta} f(x_{n+1}) - \wt{M}_{\beta}f(x_n)\big|^q}{|x_{n+1} - x_n|^{q-1} } \right)^{1/q} \leq 8^{1/q} \,\var(f).
\end{equation*}
Since the original partition $\mc{P}$ was arbitrary, we arrive at \eqref{Intro_q_bound1}.

\smallskip

We now prove \eqref{Sec5_point_conv}. Recall that $\lim_{\varepsilon \to 0} f_{\varepsilon}(x) = f(x)$ at every point of continuity of $f$ (in particular, almost everywhere). By Fatou's lemma, for any nontrivial interval $[x-s, x+r]$ containing $x$ we have
\begin{equation*}
A_{r,s}f(x) \leq \liminf_{\varepsilon \to 0} A_{r,s}f_{\varepsilon}(x) \leq \liminf_{\varepsilon \to 0} \wt{M}_{\beta} f_{\varepsilon}(x)\,,
\end{equation*}
from which we conclude that 
\begin{equation}\label{Sec5_liminf}
\wt{M}_{\beta} f(x) \leq \liminf_{\varepsilon \to 0} \wt{M}_{\beta} f_{\varepsilon}(x).
\end{equation}
We prove the opposite  inequality by contradiction. Given $x \in \R$, assume that for some $\eta >0$ we have
\begin{equation}\label{Sec5_lim_sup}
\limsup_{\varepsilon\rightarrow0}\widetilde M_{\beta}f_{\varepsilon}(x)>(1+2\eta)\,\widetilde M_{\beta}f(x)
\end{equation}
(in particular, $f \not \equiv 0$ and $\widetilde M_{\beta}f(x) >0$). Then, for a certain sequence of $\varepsilon \to 0$, there exist $y = y_\varepsilon$ and $r = r_\varepsilon >0$ such that $x \in [y - r, y+r]$ and 
\begin{equation}\label{Sec5_eta_bound}
\frac{1}{(2r)^{1-\beta}} \int_{y-r}^{y+r} f_{\varepsilon}(t)\,\dt > (1+\eta)\,\widetilde M_{\beta}f(x).
\end{equation}
Since $\|f_{\varepsilon}\|_{L^{\infty}(\R)} \leq \|f\|_{L^{\infty}(\R)}$ and we are assuming that $0 < \beta < 1$, we note that we must have $r_\varepsilon > c >0$ for this sequence of $\varepsilon \to 0$. On the other hand, observe that 
\begin{align}\label{Sec5_grande_split}
\begin{split}
\frac{1}{(2r)^{1-\beta}} \int_{y-r}^{y+r} f_{\varepsilon}(t)\,\dt &=\frac{1}{(2r)^{1-\beta}} \int_{y-r}^{y+r}\int_{-\varepsilon}^{\varepsilon}f(t-u)\,\varphi_{\varepsilon}(u)\,\du\,\dt\\
&=\frac{1}{(2r)^{1-\beta}}\int_{-\varepsilon}^{\varepsilon}\left(\int_{y-r}^{y+r}f(t-u)\,\dt\right)\varphi_{\varepsilon}(u)\,\du\\
&\leq\frac{1}{(2r)^{1-\beta}}\int_{-\varepsilon}^{\varepsilon}\left(\int_{y-r-\varepsilon}^{y+r+\varepsilon}f(t)\,\dt\right)\varphi_{\varepsilon}(u)\,\du\\
&\leq \frac{(2r + 2\varepsilon)^{1- \beta}}{(2r)^{1-\beta}}\int_{-\varepsilon}^{\varepsilon} \widetilde M_{\beta}f(x)\, \varphi_{\varepsilon}(u)\,\du\\
&\leq \frac{(2r + 2\varepsilon)^{1- \beta}}{(2r)^{1-\beta}}\widetilde M_{\beta}f(x).
\end{split}
\end{align}
From \eqref{Sec5_eta_bound} and \eqref{Sec5_grande_split} we conclude that
\begin{equation*}
\frac{(2r + 2\varepsilon)^{1- \beta}}{(2r)^{1-\beta}} > 1 + \eta.
\end{equation*}
If we restrict ourselves to the range $\varepsilon \leq 1$, the inequality above implies that $r \leq N$ for some large $N = N(\beta, \eta)$ and then $|y| \leq |x| + N$.  Therefore, there exists a subsequence $\{y_{\varepsilon_{k}},r_{\varepsilon_{k}}\}\subset\{y_{\varepsilon},r_{\varepsilon}\}$ and a pair $(y_0,r_0)$ with $r_{0}>0$ such that $y_{\varepsilon_{k}}\to y_{0}$ and $r_{\varepsilon_{k}}\to r_{0}$ as $\varepsilon_{k}\to 0$. Then $x \in [y_0 - r_0, y_0 + r_0]$ and 
\begin{equation*}
\widetilde M_{\beta}f(x) \geq \frac{1}{(2r_0)^{1-\beta}} \int_{y_0-r_0}^{y_0+r_0} f(t)\,\dt = \lim_{\varepsilon_k \to 0} \frac{1}{(2r_{\varepsilon_k})^{1-\beta}} \int_{y_{\varepsilon_k}-r_{\varepsilon_k}}^{y_{\varepsilon_k}+r_{\varepsilon_k}} f_{\varepsilon}(t)\,\dt \geq (1+\eta)\,\widetilde M_{\beta}f(x),
\end{equation*}
which is a contradiction. Hence \eqref{Sec5_lim_sup} cannot occur and we must have
\begin{equation*}
\limsup_{\varepsilon\rightarrow0}\widetilde M_{\beta}f_{\varepsilon}(x)\leq \widetilde M_{\beta}f(x),
\end{equation*}
which, together with \eqref{Sec5_liminf}, establishes the pointwise convergence \eqref{Sec5_point_conv} and concludes the proof of the inequality in \eqref{Intro_q_bound1}.

\smallskip

Finally, the fact that $\wt{M}_{\beta}f$ is absolutely continuous with 
$$\big\|\big(\wt{M}_{\beta}f\big)'\big\|_{L^q(\R)} = \var_q\big(\wt{M}_{\beta}f\big)$$
now follows from Proposition \ref{Prop_Riesz}. This concludes the proof of Theorem \ref{Thm1}.

\subsection{Proof of Proposition \ref{q variation L}}\label{Proof_Lipschitz_case} Recall that we are assuming that $f \geq 0$. If $\var(f) = \infty$, we understand that inequality \eqref{Sec5_Lem8_eq1} is true, so let us also assume that $f$ is of bounded variation. The next lemma is the continuous analogue of Lemma \ref{Lem7_crucial} and is the core of this proof.

\begin{lemma}\label{Lem10_crucial_too}
Let $0< \beta <1$. Let $f:\R \to \R^+$ be a Lipschitz function of bounded variation such that $\wt{M}_{\beta}f$ is non-constant $($in particular, $\wt{M}_{\beta}f \not\equiv \infty$$)$.
\begin{enumerate}
\item[(i)] Let $x_1 < x_2 < \ldots < x_N$ be a sequence of real numbers such that 
\begin{equation*}
\wt{M}_{\beta}f(x_1) > \wt{M}_{\beta}f(x_2)  \geq \ldots \geq \wt{M}_{\beta}f(x_{N-1}) \geq \wt{M}_{\beta}f(x_N).
\end{equation*}
Let $r,s \geq 0$ be such that $\wt{M}_{\beta}f(x_1) = A_{r,s}f(x_1)$, with $r+s$ minimal $($and then with $r$ minimal, if necessary$)$. Then
\begin{equation}\label{Sec5_Lem10_goal1}
\sum_{n=1}^{N-1} \frac{\big|\wt{M}_{\beta} f(x_{n+1}) - \wt{M}_{\beta}f(x_n)\big|^q}{|x_{n+1} - x_n|^{q-1} } \leq 4 \,\|f'\|_{L^1(\R)}^{q-1}\,\int_{x_1 - r}^{x_N} |f'(x)|\,\dx.
\end{equation}

\smallskip

\item[(ii)] Let $x_1 < x_2 < \ldots < x_N$ be a sequence of real numbers such that 
\begin{equation*}
\wt{M}_{\beta}f(x_1) \leq  \wt{M}_{\beta}f(x_2) \leq \ldots \leq \wt{M}_{\beta}f(x_{N-1}) < \wt{M}_{\beta}f(x_N).
\end{equation*}
Let $r,s \geq 0$ be such that $\wt{M}_{\beta}f(x_N) = A_{r,s}f(x_N)$, with $r+s$ minimal $($and then with $s$ minimal, if necessary$)$. Then
\begin{equation}\label{Sec5_Lem10_goal2}
\sum_{n=1}^{N-1} \frac{\big|\wt{M}_{\beta} f(x_{n+1}) - \wt{M}_{\beta}f(x_n)\big|^q}{|x_{n+1} - x_n|^{q-1} } \leq 4 \,\|f'\|_{L^1(\R)}^{q-1}\,\int_{x_1}^{x_N+s} |f'(x)|\,\dx.
\end{equation}
\end{enumerate}
\end{lemma}

\begin{proof} We prove (i) and (ii) is analogous. Note that the existence of such a pair $(r,s)$ is guaranteed by Lemma \ref{Sec4_Lem8_prelim} and the continuity properties of the averages. For such pair $(r,s)$, since $\wt{M}_{\beta}f(x_1) > \wt{M}_{\beta}f(x_N)$, we must have $s < x_N - x_1$. Let $a \in[x_1-r,x_1+s]$ be a point such that
 \begin{equation*}
 \frac{1}{(r+s)}\int_{x_1-r}^{x_1+s}f(x)\,\dx = f(a),
 \end{equation*} 
and let $b \in [x_N-r-s,x_N]\subset[x_1-r,x_N]$ be a point such that
\begin{equation*}
\frac{1}{r+s}\int_{x_N-r-s}^{x_N}f(x)\,\dx = f(b).
\end{equation*} 
We then obtain
\begin{align}\label{Sec4_Lem10_eq1_initial_est}
\begin{split}
\big|\wt{M}_{\beta}f(x_1)- \wt{M}_{\beta}f(x_N)\big|&\leq \big|A_{r,s}f(x_1) - A_{r+s,0}f(x_N)\big|\\
&= (r+s)^{\beta}\, |f(a) - f(b)|\\
&\leq (r+s)^{\beta}\int_{x_1-r}^{x_N}|f'(x)|\,\dx.
\end{split}
\end{align}

\smallskip

Now let $m$ be the smallest integer with $1 \leq m \leq N-1$ such that 
$$\frac{\widetilde M_{\beta}f(x_m)-\widetilde M_{\beta}f(x_{m+1})}{|x_{m} - x_{m+1}|} =  \max\left\{\frac{\widetilde M_{\beta}f(x_n)-\widetilde M_{\beta}f(x_{n+1})}{|x_{n} - x_{n+1}|}; \,1 \leq n \leq N-1 \right\} > 0.$$
Let $t,u \geq 0$ be such that $\wt{M}_{\beta}f(x_m) = A_{t,u}f(x_m)$, with $t+u$ minimal $($and then with $t$ minimal, if necessary$)$. The existence of such a pair $(t,u)$ is guaranteed by Lemma \ref{Sec4_Lem8_prelim} and we have $0\leq u < x_{m+1}-x_{m}$.

\subsubsection*{Step 1} Let us first consider the case when $t + u \geq r+s$. Using \eqref{Sec4_Lem10_eq1_initial_est} we have
\begin{align*}
\sum_{n=1}^{N-1} & \frac{\big|\wt{M}_{\beta} f(x_{n}) - \wt{M}_{\beta}f(x_{n+1})\big|^q}{|x_{n} - x_{n+1}|^{q-1} }\\
&\leq \left|\frac{\wt{M}_{\beta}f(x_{m})- \wt{M}_{\beta}f(x_{m+1})}{x_{m}-x_{m+1}}\right|^{q-1}\sum_{n=1}^{N-1}\big|\wt{M}_{\beta}f(x_n)- \wt{M}_{\beta}f(x_{n+1})\big|\\
&=\left|\frac{\wt{M}_{\beta}f(x_{m})-\wt{M}_{\beta}f(x_{m+1})}{x_{m}-x_{m+1}}\right|^{q-1}\big|\wt{M}_{\beta}f(x_1)- \wt{M}_{\beta}f(x_N)\big|\\
&\leq\left|\frac{\wt{M}_{\beta}f(x_{m})-\wt{M}_{\beta}f(x_{m+1})}{x_{m}-x_{m+1}}\right|^{q-1}(r+s)^{\beta}\int_{x_1-r}^{x_N}|f'(x)|\,\dx\\
&=\left|(r+s)^{1-\beta}\, \frac{\wt{M}_{\beta}f(x_{m})-\wt{M}_{\beta}f(x_{m+1})}{x_{m}-x_{m+1}}\right|^{q-1}\int_{x_1-r}^{x_N}|f'(x)|\,\dx\\
&\leq\left|\left(\frac{{r+s}}{t+u}\right)^{1-\beta}\int_{x_{m}-t}^{x_{m}+u}\frac{f(y)-f(y+x_{m+1}-x_m - u)}{x_{m+1}-x_{m}}\,\dy\right|^{q-1}\int_{x_1-r}^{x_N}|f'(x)|\,\dx\\
&\leq\left|\left(\frac{{r+s}}{t+u}\right)^{1-\beta}\frac{1}{x_{m+1}-x_{m}}\int_{x_{m}-t}^{x_{m}+u}\int_{0}^{x_{m+1}-x_{m}-u}|f'(y+z)|\,\dz\,\dy\right|^{q-1}\int_{x_1-r}^{x_N}|f'(x)|\,\dx\\
&\leq\left|\left(\frac{{r+s}}{t+u}\right)^{1-\beta}\frac{1}{x_{m+1}-x_{m}}\int_{0}^{x_{m+1}-x_{m}-u}\int_{x_{m}-t}^{x_{m}+u}|f'(y+z)|\,\dy\,\dz\right|^{q-1}\int_{x_1-r}^{x_N}|f'(x)|\,\dx\\
&\leq\|f'\|_{L^{1}(\R)}^{q-1}\int_{x_1-r}^{x_N}|f'(x)|\,\dx\,,
\end{align*}
which establishes \eqref{Sec5_Lem10_goal1} in this case.

\subsubsection*{Step 2} We now consider the case when $t + u < r+s$ and $x_m - t \leq x_1$. In this case note that $x_1-r \leq x_m-t$ (which is clear if $m=1$, and if $m>1$ we note that $s < x_m - x_1$). Reasoning as before we have
\begin{align*}
& \sum_{n=1}^{N-1}  \frac{\big|\wt{M}_{\beta} f(x_{n}) - \wt{M}_{\beta}f(x_{n+1})\big|^q}{|x_{n} - x_{n+1}|^{q-1} }\\
&\leq  \left|\frac{\wt{M}_{\beta}f(x_{m})- \wt{M}_{\beta}f(x_{m+1})}{x_{m}-x_{m+1}}\right|^{q}\,\sum_{n=1}^{m-1}|x_n- x_{n+1}|  \\
& \ \ \ \ \ \ \ \ \ \ + \left|\frac{\wt{M}_{\beta}f(x_{m})- \wt{M}_{\beta}f(x_{m+1})}{x_{m}-x_{m+1}}\right|^{q-1}\,\sum_{n=m}^{N-1}\big|\wt{M}_{\beta}f(x_n)- \wt{M}_{\beta}f(x_{n+1})\big|\\
&\leq  \left|\frac{\wt{M}_{\beta}f(x_{m})- \wt{M}_{\beta}f(x_{m+1})}{x_{m}-x_{m+1}}\right|^{q} \,|x_1 - x_m|\\
& \ \ \ \ \ \ \ \ \ \ + \left|\frac{\wt{M}_{\beta}f(x_{m})- \wt{M}_{\beta}f(x_{m+1})}{x_{m}-x_{m+1}}\right|^{q-1}\big|\wt{M}_{\beta}f(x_m)- \wt{M}_{\beta}f(x_{N})\big|\\
&\leq  t\, \left|\frac{\wt{M}_{\beta}f(x_{m})- \wt{M}_{\beta}f(x_{m+1})}{x_{m}-x_{m+1}}\right|^{q} \\
& \ \ \ \ \ \ \ \ \ \ + \left|\frac{\wt{M}_{\beta}f(x_{m})- \wt{M}_{\beta}f(x_{m+1})}{x_{m}-x_{m+1}}\right|^{q-1} (t+u)^{\beta}\int_{x_m-t}^{x_N}|f'(x)|\,\dx.\\
& \leq \left|\left(\frac{{t}}{t+u}\right)^{1-\beta}\int_{x_{m}-t}^{x_{m}+u}\frac{f(y)-f(y+x_{m+1}-x_{m} - u)}{x_{m+1}-x_{m}}\,\dy\right|^{q}\\
& \ \ \ \ \ \ \ \ \ \ + \left|\int_{x_{m}-t}^{x_{m}+u}\frac{f(y)-f(y+x_{m+1}-x_{m}-u)}{x_{m+1}-x_{m}}\,\dy\right|^{q-1}\,\int_{x_m-t}^{x_N}|f'(x)|\,\dx\\
& \leq \left|\frac{1}{x_{m+1} - x_m}\int_{x_{m}-t}^{x_{m}+u}\int_{0}^{x_{m+1}-x_{m}-u}|f'(y+z)|\,\dz\,\dy\right|^{q}\\
& \ \ \ \ \ \ \ \ \ \ + \left|\frac{1}{x_{m+1} - x_m}\int_{x_{m}-t}^{x_{m}+u}\int_{0}^{x_{m+1}-x_{m}-u}|f'(y+z)|\,\dz\,\dy\right|^{q-1}\,\int_{x_m-t}^{x_N}|f'(x)|\,\dx\\
& = \left|\frac{1}{x_{m+1} - x_m}\int_{0}^{x_{m+1}-x_{m}-u}\int_{x_{m}-t}^{x_{m}+u}|f'(y+z)|\,\dz\,\dy\right|^{q}\\
& \ \ \ \ \ \ \ \ \ \ + \left|\frac{1}{x_{m+1} - x_m}\int_{0}^{x_{m+1}-x_{m}-u}\int_{x_{m}-t}^{x_{m}+u}|f'(y+z)|\,\dz\,\dy\right|^{q-1}\,\int_{x_m-t}^{x_N}|f'(x)|\,\dx\\
& \leq  \|f'\|_{L^{1}(\R)}^{q-1}\ \left\{\frac{1}{x_{m+1} - x_m}\int_{0}^{x_{m+1}-x_{m}-u}\int_{x_{m}-t}^{x_{N}}|f'(x)|\,\dx\,\dy + \int_{x_m-t}^{x_N}|f'(x)|\,\dx\right\}\\
& \leq 2  \, \|f'\|_{L^{1}(\R)}^{q-1} \int_{x_m-t}^{x_N}|f'(x)|\,\dx\\
& \leq 2  \, \|f'\|_{L^{1}(\R)}^{q-1} \int_{x_1-r}^{x_N}|f'(x)|\,\dx,
\end{align*}
which establishes \eqref{Sec5_Lem10_goal1}.

\subsubsection*{Step 3} Finally, we consider the case  when $t + u < r+s$ and $x_1 < x_m - t$. We let $k$ be the unique integer such that $x_{k-1} < x_m - t \leq x_k$. Arguing as in Step 2 we get 
\begin{align}\label{Sec5_cont_ind_1}
& \sum_{n=k}^{N-1}  \frac{\big|\wt{M}_{\beta} f(x_{n}) - \wt{M}_{\beta}f(x_{n+1})\big|^q}{|x_{n} - x_{n+1}|^{q-1} } \leq 2  \, \|f'\|_{L^{1}(\R)}^{q-1} \int_{x_{k-1}}^{x_N}|f'(x)|\,\dx .
\end{align}
We then proceed inductively. Let $(m_1, t_1, u_1, k_1) = (m,t,u,k)$. Having defined $(m_1, t_1, u_1, k_1), (m_2, t_2, u_2, k_2),$ $\ldots,$ $(m_{l-1}, t_{l-1}, u_{l-1}, k_{l-1})$, if $t_{l-1} + u_{l-1} < r+s$ and $x_1 < x_{m_{l-1}} - t_{l-1}$ then $k_{l-1}$ is the unique integer such that $x_{k_{l-1}-1} < x_{m_{l-1}} - t_{l-1} \leq x_{k_{l-1}}$. We then let $m_l$ be the smallest integer with $1 \leq m_l \leq k_{l-1}-1$ such that 
$$\frac{\widetilde M_{\beta}f(x_{m_l})-\widetilde M_{\beta}f(x_{m_l+1})}{|x_{m_l} - x_{m_l+1}|} =  \max\left\{\frac{\widetilde M_{\beta}f(x_n)-\widetilde M_{\beta}f(x_{n+1})}{|x_{n} - x_{n+1}|}; \,1 \leq n \leq k_{l-1}-1 \right\} > 0.$$
Let $t_l,u_l \geq 0$ be such that $\wt{M}_{\beta}f(x_{m_l}) = A_{t_l,u_l}f(x_{m_l})$, with $t_l+u_l$ minimal $($and then with $t_l$ minimal, if necessary$)$. The existence of such a pair $(t_l,u_l)$ is guaranteed by Lemma \ref{Sec4_Lem8_prelim}. If $t_{l} + u_{l} < r+s$ and $x_1 < x_{m_{l}} - t_{l}$ we let $k_l$ be the unique integer such that $x_{k_{l}-1} < x_{m_{l}} - t_{l} \leq x_{k_{l}}$. We then reboot Step 3 to obtain
\begin{align}\label{Sec5_cont_ind_2}
& \sum_{n=k_l}^{k_{l-1}-1}  \frac{\big|\wt{M}_{\beta} f(x_{n}) - \wt{M}_{\beta}f(x_{n+1})\big|^q}{|x_{n} - x_{n+1}|^{q-1} } \leq 2  \, \|f'\|_{L^{1}(\R)}^{q-1} \int_{x_{{k_l}-1}}^{x_{k_{l-1}}}|f'(x)|\,\dx .
\end{align}
This process must terminate, i.e. there is a smallest integer $L$ such that either: (i) $t_{L} + u_{L} \geq r+s$ or (ii) $t_{L} + u_{L} < r+s$ and $x_{m_{L}} - t_{L} \leq x_1$. In the first case (resp. second case) we use Step 1 (resp. Step 2) to bound the $q$-variation of $\wt{M}_{\beta}f$ (over the partition $\{x_n\}$) from $x_1$ to $x_{k_{L-1}}$. We then sum all the previous inequalities \eqref{Sec5_cont_ind_1} and \eqref{Sec5_cont_ind_2} to arrive at the desired conclusion (note that the sum of the integrals on right-hand sides of \eqref{Sec5_cont_ind_1} and \eqref{Sec5_cont_ind_2} has a two-fold overlap over each interval $[x_{{k_l}-1}, x_{k_l}]$).
\end{proof}

\begin{proof}[Proof of Proposition  \ref{q variation L}] Fix a partition $\mc{P} = \{x_1 < x_2 < \ldots < x_N\}$. For a generic function $g:\R \to \R$, we say that an interval $[x_n,x_m]$ is a {\it string of local maxima} of $g$ (relative to the partition $\mc{P}$)  if 
\begin{equation}\label{Sec5_string}
g(x_{n-1}) < g(x_n) = \ldots = g(x_m) > g(x_{m+1}),
\end{equation}
provided $n \neq 1$ and $m \neq N$. In case $n=1$ (resp. $m= N$) we disregard the leftmost (resp. rightmost) inequality in \eqref{Sec5_string}. When $m \neq N$, the rightmost point $x_m$ of such a string is a {\it right local maximum} of $g$, and when $n\neq 1$ the leftmost point $x_n$ is a {\it left local maximum} of $g$. We define {\it string of local minima}, {\it right local minimum} and {\it left local minimum} (relative to the partition $\mc{P}$) analogously.

\smallskip

Assume that $\wt{M}_{\beta}f$ is not constant (in case $\wt{M}_{\beta}f$ is constant the result is obviously true). The strategy is the same as in the discrete case, to bound the $q$-variation of $\wt{M}_{\beta}f$ between a right local maximum and the next left local minimum. The $q$-variation of $\wt{M}_{\beta}f$ between a right local minimum and the next left local maximum is treated analogously.

\smallskip

Let $x_{m}$ be a right local maximum and $x_{l}$ be the next left local minimum of $\wt{M}_{\beta}f$ (i.e. with the smallest $l>m$). Let $r,s \geq 0$ be such that $\wt{M}_{\beta}f(x_{m}) = A_{r,s}f(x_{m})$, with $r+s$ minimal (and then with $r$ minimal, if necessary). By Lemma \ref{Lem10_crucial_too} we have
\begin{equation}\label{Sec5_proof_Prop9_fund_eq}
\sum_{n=m}^{l-1} \frac{\big|\wt{M}_{\beta} f(x_{n+1}) - \wt{M}_{\beta}f(x_n)\big|^q}{|x_{n+1} - x_n|^{q-1} } \leq 4 \,\|f'\|_{L^1(\R)}^{q-1}\,\int_{x_m - r}^{x_l} |f'(x)|\,\dx.
\end{equation}
In case there exists a right local minimum $x_{l'}$ before $x_m$ (which is the case if $x_m$ is not the first right local maximum), note that we have $x_{l'} < x_m - r$. We may therefore sum the inequalities \eqref{Sec5_proof_Prop9_fund_eq} over all pairs of consecutive local extrema (the non-increasing and non-decreasing pieces). Noting the additional two-fold overlap on the right-hand side we arrive at 
\begin{equation*}
\sum_{n=1}^{N-1} \frac{\big|\wt{M}_{\beta} f(x_{n+1}) - \wt{M}_{\beta}f(x_n)\big|^q}{|x_{n+1} - x_n|^{q-1} } \leq 8 \,\|f'\|_{L^1(\R)}^{q}.
\end{equation*}
Since the right-hand side is now independent of the partition $\mc{P}$, we may take the supremum over all such partitions to obtain the desired result.
\end{proof}

\section*{Acknowledgents}
\noindent E.C. acknowledges support from CNPq-Brazil grants $305612/2014-0$ and $477218/2013-0$, and FAPERJ grant $E-26/103.010/2012$. J. M. acknowledges support from CAPES-Brazil. We credit the main idea for the proof of the boundedness part of Theorem \ref{Thm3} (i) to Prof. Keith Rogers (ICMAT - Madrid), to whom we are greatly thankful. We are also greatly thankful to Prof. Martin Lind for bringing the classical result of F. Riesz on the $q$-variation (Proposition \ref{Prop_Riesz}) and references \cite{BaLi,N} to our attention.

\end{document}